\documentclass{amsart}

\usepackage{amssymb}
\usepackage[all]{xy}
\usepackage{tikz-cd}

\usepackage{graphicx}
\usepackage[colorlinks=true]{hyperref} 

\theoremstyle{plain}
\newtheorem{lem}{Lemma}[section]
\newtheorem{cor}[lem]{Corollary}
\newtheorem{prop}[lem]{Proposition}
\newtheorem{thm}[lem]{Theorem}

\theoremstyle{definition}
\newtheorem{ex}[lem]{Example}
\newtheorem{rem}[lem]{Remark}
\newtheorem{dfn}[lem]{Definition}

\newcommand{\la}{\longrightarrow}
\newcommand{\Z}{\mathbb{Z}}      

\newcommand{\im}{\operatorname{im}}            
\newcommand{\End}{\operatorname{End}}  
\newcommand{\Hom}{\operatorname{Hom}}  

\DeclareMathOperator{\pt}{pt} 
\newcommand{\codim}{\operatorname{codim}}            
\newcommand{\Os}{\mathcal{O}} 
\newcommand{\PP}{\operatorname{\mathbb{P}}} 
\newcommand{\Sch}{\operatorname{Sch}}  
\newcommand{\Sm}{\operatorname{Sm}} 
\newcommand{\Ll}{\mathcal{L}} 
\newcommand{\Fl}{\text{Fl}} 

\DeclareMathOperator{\res}{res} 

\DeclareMathOperator{\Atype}{A} 
\DeclareMathOperator{\Dtype}{D} 

\newcommand{\hh}{\mathtt{h}} 
\DeclareMathOperator{\CH}{CH} 

\begin{document}

\title[Relative equivariant motives]{Relative equivariant motives and modules}

\author[B. Calm\`es]{Baptiste Calm\`es}
\author[A. Neshitov]{Alexander Neshitov}
\author[K. Zainoulline]{Kirill Zainoulline}

\address{Baptiste Calm\`es, 
Facult\'e des Sciences Jean Perrin, Universit\'e d'Artois, Rue Jean Souvraz SP 18, 62307 Lens Cedex, France}

\address{Alexander Neshitov, Department of Mathematics, Western University, Middlesex College, London ON N6A 5B7, Canada}

\address{Kirill Zainoulline, 
Department of Mathematics and Statistics, University of Ottawa, 150 Louis-Pasteur, Ottawa ON K1N 6N5, Canada}



\begin{abstract} We introduce and study various categories of (equivariant) motives of (versal) flag varieties.
We relate these categories with certain categories of parabolic (Demazure) modules. 
We show that the motivic decomposition type of a versal flag variety depends 
on the direct sum decomposition type of the parabolic module. To do this we use 
localization techniques of Kostant-Kumar in the context of generalized oriented cohomology as well as
the Rost nilpotence principle for algebraic cobordism and its generic version.
As an application, we obtain new proofs and examples of indecomposable Chow motives of versal flag varieties.
\end{abstract}

\maketitle

\tableofcontents

\section{Introduction}

The theory of Chow motives of twisted flag varieties has been a topic of intensive investigations for decades. Inspired by results on motives of quadrics by Rost~\cite{Ro98} and Vishik~\cite{Vi00} and on motives of Severi-Brauer varieties by Karpenko \cite{Ka96}, it has developed into a powerful tool to study quadratic forms, linear algebraic groups and the associated homogeneous spaces over arbitrary fields (see~\cite{EKM} for applications to the theory of quadratic forms). Several important invariants of algebraic groups, e.g. canonical dimension, can be interpreted using the language of motives (see~\cite{PSZ}).

We fix a split reductive linear algebraic group $G$ over a field $k$ of characteristic $0$, its split maximal torus $T$ and the Borel subgroup $B\supset T$. In the present paper we focus on the study of motivic decompositions of versal flag varieties $E/P$, where $E$ is a versal $G$-torsor and $P\supset B$ is a standard parabolic subgroup of $G$. We refer to \cite{Ka17} for the recent discussion concerning Chow groups and K-theory of versal flags and to \cite{PS} and \cite{NPSZ} for recent results about motivic decompositions. It was shown in \cite{NPSZ} that direct sum decompositions of the motive of complete versal flag $E/B$ correspond to direct sum decompositions of the $D$-module $D^\star$, where $D$ is the associated affine Hecke-type algebra for $G$ and $D^\star$ is the $T$-equivariant cohomology of $G/B$. In the PhD thesis of the second author \cite{Ne16} this result was extended to all versal $E/P$'s, where $P$ is special (all $P$-torsors are locally trivial in Zariski topology). Our goal is to push these results further by investigating motives of $E/P$'s for arbitrary parabolic subgroups $P$.

We work in a bit more general situation than the theory of Chow motives. Namely we consider an oriented cohomology theory $\hh$ in the sense of Levine-Morel~\cite{LM} and the theory of the associated $\hh$-motives (see~\cite{GV} for definitions and basic properties of this category). As the first step, we establish several connections between the following pseudoabelian categories: $\hh$-motives $\mathfrak{M}_k$ of versal flags, $G$-equivariant motives $\mathfrak{M}_G$ and relative equivariant motives $\mathfrak{M}_{G|T}$ of split flags. More precisely, we show that there is a chain of 1-1 correspondences between the direct sum decompositions of the respective objects
\begin{equation}\tag{*}
[\text{versal }E/P] \in \mathfrak{M}_k \stackrel{\text{Thm.}\ref{prop:maincor}}\longleftrightarrow [G/P] \in \mathfrak{M}_G \stackrel{\text{Cor.}\ref{cor:GT}}\longleftrightarrow [G/P] \in \mathfrak{M}_{G|T}.
\end{equation}
To do this we use the generalization (Lemma~\ref{lem:hintersect}) from Chow groups to algebraic cobordism of \cite[Lemma~6.2]{VZ} and the Rost nilpotence for free theories proven in~\cite{GV}.

As the second step, we introduce the category of $W$-equivariant motives $\mathfrak{M}_T^W$, where $W$ is the Weyl group of $G$ with respect to $T$ and show (Corollary~\ref{cor:mainGP}) that there is the natural inclusion of endomorphism rings
\begin{equation}\tag{**}
\End_{\mathfrak{M}_{G|T}}([G/P]) \hookrightarrow \End_{D}(D_P^\star),
\end{equation}
where $D_P^\star$ is the so called parabolic $D$-module with respect to the $\odot$-action of~\cite[\S3]{LZZ}. (In the context of the usual cohomology $D_P^\star$ coincides with $H_T(G/P)$ and the $\odot$-action is simply the induced $W$-action studied by Brion, Knutson, Peterson, Tymoczko and others.) Informally speaking, (*) and (**) say that the motivic decomposition type of $[E/P]$ is bounded by the direct sum decomposition type of $D_P^\star$. In particular, if $D_P^\star$ is indecomposable (as a $D$-module), then so is $[E/P]$ (as a motive).

Next, we investigate the endomorphism ring $E_D=\End_{D}(D_P^\star)$ using the localization techniques of~\cite{CZZ1} and~\cite{CZZ2} (these are generalizations of the respective techniques for Chow groups and $K$-theory of Kostant-Kumar~\cite{KK86, KK90}). We describe endomorphisms of $E_D$ with respect to the localization basis (Lemma~\ref{lem:condd}) and the Schubert basis for the Chow theory (Corollary~\ref{cor:res}). We relate $E_D$ with the induced (integral/modular) representations of the Weyl group $W$ and its parabolic subgroup $W_P$. In particular, we show that if $D_P^\star$ is indecomposable, then so is $\mathrm{Ind}^W_{W_P} \mathbf{1}$ (Corollary~\ref{cor:integr}).

Finally, we apply all these results to obtain new proofs and examples of indecomposable motives of versal flag varieties. 
The most interesting example here is the 6-dimensional involution variety, 
the form of a split projective quadric twisted by a generic $HSpin_8$-torsor (see Remark~\ref{rem:hspin}).

The paper is organized as follows. First, we recall basic properties of equivariant oriented cohomology and of the associated category of equivariant motives; we introduce and study the categories of relative equivariant motives and $W$-equivariant motives. Then we investigate relations between relative equivariant motives and parabolic modules. In Section~\ref{sec:localized} we study the endomorphisms of parabolic modules using the localization techniques of \cite{CZZ1} amd \cite{CZZ2}. In the next section we study the same endomorphisms for the Chow theory with respect to the Schubert basis using the results of \cite{LZZ}. In the last section we discuss applications to Chow motives of versal flags (Sever-Brauer varieties, 4 and 6-dimensional quadrics and involution varieties). In the appendix we prove the generic nilpotence for algebraic cobordism.    

\ 

\paragraph{\it Acknowledgements} 
The first author acknowledges the support of the French Agence Nationale de la Recherche (ANR) under reference ANR-12-BL01-0005. 
The last author was partially supported by the NSERC Discovery Grant.


\section{Categories of equivariant motives}

In this section we recall basic properties of equivariant oriented cohomology and of the associated category of equivariant motives. We introduce and study the categories of relative equivariant motives and $W$-equivariant motives.

\ 

\paragraph{\it Equivariant cohomology.} 
Given a split reductive algebraic group $G$ over a field $k$ of characteristic~$0$ and given a graded oriented cohomology theory $\hh$ of Levine-Morel~\cite[\S1]{LM}, let $\hh_G$
be the $G$-equivariant oriented cohomology theory in the sense of \cite[\S2]{CZZ2} defined on the category of smooth quasi-projective $G$-varieties over $k$. A~key property of $\hh$ and of $\hh_G$ is the Quillen formula \cite[Lemma~1.1.3]{LM}
\[
c_1^{\hh}(\Ll_1\otimes \Ll_2)=F(c_1^{\hh}(\Ll_1),c_1^{\hh}(\Ll_2)),
\]
where $c_1^\hh$ is the first (equivariant) characteristic class in the theory $\hh$ (resp.~$\hh_G$), $\Ll_i$ is a (equivariant) line bundle over some ($G$-) variety $X$ and $F(x,y)\in R[[x,y]]$ is the associated formal group law over the coefficient ring $R=\hh(\pt)$ (here $\pt=\mathrm{Spec}\,k$). 

Conversely, given a formal group law $F$ over $R$ one defines the algebraic oriented cohomology theory as
\[
\hh(-):=\Omega(-)\otimes_{\Omega(\pt)}R,
\] 
where $\Omega$ is the algebraic cobordism of Levine-Morel (universal cohomology theory) over the Lazard ring $\mathbb{L}=\Omega(\pt)$ and the map $\Omega(\pt)\to R$ is determined by $F$. Such theories $\hh$ are called free theories  \cite[\S2.12]{GV}. We refer to \cite{GV} for the basic facts and properties of free theories.

By the following process (the Borel construction) one can produce many examples of equivariant oriented theories, such as equivariant Chow theory~\cite{To,EG}, $K$-theory~\cite{Me} and equivariant algebraic cobordism~\cite{HML} (see also \cite[\S2]{CZZ2} and \cite[\S3]{NPSZ} for more details):

Consider a system of $G$-representations $V_i$ and its open subsets $U_i\subseteq V_i$ such that
\begin{itemize}
\item $G$ acts freely on $U_i$ and the quotient $U_i/G$ exists as a variety over $k$,
\item $V_{i+1}=V_i\oplus W_i$ for some representation $W_i$,
\item $U_i\subseteq U_i\oplus W_i\subseteq U_{i+1}$, and $U_i\oplus W_i\to U_{i+1}$ is an open inclusion, and
\item  $\codim(V_i\setminus U_i)$ strictly increases.
\end{itemize}
Such a system is called a good system of representations of $G$. 

Now let $X$ be a smooth $G$-variety (e.g.~a projective homogeneous variety). Following~\cite[\S3 and \S5]{HML} the inverse limit induced by pull-backs 
\[
\varprojlim_i\; \hh(U_i\times^G X),\quad U_i \times^G X=(U_i \times X)/G,
\] 
does not depend on the choice of the system $(V_i,U_i)$ and, hence, defines the equivariant oriented cohomology $\hh_G(X)$.

\paragraph{\it Equivariant motives.} 
Fix a Borel subgroup $B$ of $G$. Consider the full additive subcategory of the category of $G$-equivariant $\hh$-motives generated by the motives (and their direct summands) of flag varieties $G/P$ for all standard parabolic subgroups $P$ of $G$ containing $B$. We refer to \cite[\S2]{GV} and \cite{PS} for definitions and basic properties of this category.  Recall only that morphisms in this pseudo-abelian category are given by cohomology classes in $\hh_G(X\times Y)$, where $G$ acts diagonally on the product of smooth projective $G$-varieties $X$ and $Y$ and the composition is given by the usual correspondence product:
\begin{equation}\label{eq:corrs}
\beta\circ\alpha:=p_{XZ*}(p_{XY}^*(\alpha)p_{YZ}^*(\beta)),\quad \alpha \in \hh_G(X\times Y),\; \beta \in \hh_G(Y\times Z)
\end{equation}
(here $p_{XY}$ is the projection $X\times Y\times Z \to X\times Y$ and $p_{XY}^*$, $p_{XY*}$ is the induced equivariant pull-back and push-forward respectively).

Taking the respective graded components (e.g.~$\dim X$ component of $\hh_G(X\times X)$ for irreducible $X$) we obtain the category of graded equivariant motives which we call simply \emph{$G$-equivariant motives} and denote by $\mathfrak{M}_G$. We refer to \cite[\S63,64]{EKM} for the comparison between graded and non-graded motives in case $\hh=\CH$.

To simplify the notation we will omit the grading degree for the cohomology, i.e.\  we will write $\hh_G(X\times X)$ instead of $\hh_G^{\dim X}(X\times X)$. Observe that the actual grading can be easily traced down using the fact that the pull-backs preserve the codimension (degree) $m$ and the push-forwards preserve the dimension $\dim X-m$ (codegree) of cohomology classes.

\paragraph{\it Relative equivariant motives} 
Fix a split maximal torus $T\subset B$. Similarly, let $\mathfrak{M}_T$ denote the category generated by $T$-equivariant motives of $G/P$'s, i.e.\  morphisms in $\mathfrak{M}_T$ are given by classes in $\hh_T(X\times Y)$. Since the forgetful map 
\[
\res^G_T\colon \hh_{G}(X\times Y) \la \hh_{T}(X\times Y)
\] 
commutes with pull-backs and push-forwards, it induces the forgetful functor 
\[
\mathcal{R}es^G_T\colon \mathfrak{M}_G \la \mathfrak{M}_T.
\] 
\begin{dfn}
The categorical image of $\mathcal{R}es^G_T$, i.e.\  the wide subcategory of $\mathfrak{M}_T$ with morphisms given by classes in $\im(\res^G_T)$, is called the category of \emph{relative equivariant motives} and denoted by $\mathfrak{M}_{G|T}$.
\end{dfn}
We denote the endomorphism ring of $[X]$ in $\mathfrak{M}_{G|T}$ (resp. in $\mathfrak{M}_G$) by $\mathcal{E}_{G|T}(X)$ (resp. $\mathcal{E}_G(X)$). By definition 
\[
\mathcal{E}_{G|T}(X)=\im\big(\res^G_T\colon \hh_G(X\times X) \la \hh_T(X\times X)\big)
\] 
is an algebra (with multiplication given by the correspondence product) over the commutative ring 
\[
\mathcal{E}_{G|T}(\pt)=\im\big(\res^G_T\colon \hh_G(\pt) \la \hh_T(\pt)\big)
\]
To simplify the notation we denote $\mathcal{E}_{G|T}(\pt)$ by $\mathcal{E}_{G|T}$ and $\mathcal{E}_G(\pt)=\hh_G(\pt)$ by $\mathcal{E}_G$.

\begin{ex} 
Consider the variety of complete flags $G/B$. Since the forgetful map is injective by \cite[Lemma~4.5]{NPSZ}, we can identify $\mathcal{E}_{G}(G/B)$ with the convolution ring $(\hh_G(G/B\times G/B),\circ) \simeq (\hh_{T}(G/B),\circ)$ of \cite[\S4]{NPSZ}. 
\end{ex}

\paragraph{\it The $W$-action and motives}
Let $W$ be the Weyl group of $G$ with respect to $T$ and let $X$ be a $G$-variety. 
There is a natural (left) action of $W$ on $\hh_T(X)$. It can be either realized by pull-backs induced by the right action of $W$ on each step of the Borel construction $U\times^T X$ via
\[
(u,x)T\cdot\sigma T=(u\sigma,\sigma^{-1}x)T, \quad \sigma \in N_G(T),
\] 
where $U$ is taken to have a right $G$-action; or through the natural isomorphism $\hh_T(X)\simeq \hh_G(G/T\times X)$ and the $G$-equivariant right action of $W$ on the variety $G/T\times X$ given on points by $(gT,x)\cdot\sigma T=(gT\sigma,x)$. 

The forgetful map $\res_T^G\colon \hh_G(X) \to \hh_T(X)$ factors through the $W$-invariants $\hh_T(X)^W$ by definition. 

\begin{lem} Consider the diagonal action of $G$ on the products $X\times Y$ of $G/P$'s and, hence, the induced
action of $W$ on $\hh_T(X\times Y)$.
For any $\alpha \in \hh_T(X\times Y)^W$ and $\beta\in \hh_T(Y\times Z)^W$ we have $\beta\circ\alpha\in \hh_T(X\times Z)^W$ in $\mathfrak{M}_T$.
\end{lem}

\begin{proof}
It follows from the fact that the projections $p_{XY}$, $p_{YZ}$ and $p_{XZ}$ in the definition of the correspondence product \eqref{eq:corrs} are $W$-equivariant.
\end{proof}

\begin{dfn}
The subcategory of $\mathfrak{M}_T$ with morphisms given by correspondences from $\hh_T(X\times Y)^W$ of the lemma  
is called the category of \emph{$W$-equivariant motives} and it is denoted by $\mathfrak{M}_T^W$.
\end{dfn}
Observe that the forgetful functor factors as
\[
\mathcal{R}es_T^G\colon\mathfrak{M}_G \stackrel{\mathcal{R}es_{G|T}}\la \mathfrak{M}_{G|T} \stackrel{\mathcal{R}es^W_T}\la \mathfrak{M}_T^W \stackrel{\mathcal{R}es_T}\la \mathfrak{M}_T,
\]
where $\mathcal{R}es_{G|T}$ is full and $\mathcal{R}es^W_T$, $\mathcal{R}es_T$ are faithful by definition.


\section{Equivariant motives and parabolic modules}\label{sec:eqmot}

In this section we investigate relations between relative equivariant motives and parabolic modules. 
Our main result here is Corollary~\ref{cor:mainGP}.

\ 

\paragraph{\it Realization functor}
Consider the contravariant motivic realization functor
\[
\mathcal{R}_T\colon \mathfrak{M}_T \la S\text{-}mod 
\]
to the category of $S$-modules, where $S=\hh_T(\pt)$ is the equivariant coefficient ring. 
It is given on objects by $[X] \mapsto \hh_T(X)$ and on morphisms as
\[
\alpha \in \hh_T(X\times Y)\; \mapsto\; \alpha_*\colon \hh_T(Y) \stackrel{p_{Y}^*}\to \hh_T(X\times Y)\stackrel{\alpha\cdot }\to \hh_T(X\times Y) \stackrel{p_{X*}}\to \hh_T(X).
\]
Restricting to $W$-equivariant correspondences we obtain the realization functor
\[
\mathcal{R}_T^W\colon \mathfrak{M}_T^W \la S_W\text{-}mod,
\] 
from $\mathfrak{M}_T^W$ to the category of $S_W$-modules, where $S_W$ is the twisted group algebra of $W$.
The latter is the free left $S$-module with basis $\delta_w$, $w\in W$ 
and multiplication given by the twisted product
\[
(s \delta_w) \cdot (s' \delta_{w'})= sw(s')\delta_{ww'}
\]
(here the $W$-action is given by $\delta_w(a)=w(a)$, $w\in W$, $a\in \hh_T(X)$).

\begin{lem}
There is a commutative diagram of faithful functors
\[
\xymatrix{
 \mathfrak{M}_T^W \ar[r]^{\mathcal{R}es_T}\ar[d]_{\mathcal{R}_{T}^W} & \mathfrak{M}_T\ar[d]^{\mathcal{R}_{T}} \\
 S_W\text{-}mod \ar[r]^{res_T} & S\text{-}mod
}
\]
where $res_{T}$ is induced by the ring inclusion $S\hookrightarrow S_W$, $s\mapsto s\delta_1$. 
\end{lem}

\begin{proof} The functors $\mathcal{R}es_T$ and $res_T$ are faithful by definition.

Since all $G/P$'s are $T$-equivariant cellular spaces, 
the realization maps on morphisms $\alpha \mapsto \alpha_*$ are K\"unneth isomorphisms of \cite[Lemma~3.7]{NPSZ}. 
Hence, $\mathcal{R}_T$ is faithful (for a correspondence $\alpha$ in $\mathfrak{M}_T$, $\alpha_*=0$ $\Longrightarrow$ $\alpha=0$) and so that
$\mathcal{R}_T^W$.
\end{proof}

\begin{cor}\label{cor:inclus}
The composite of functors 
\[
\mathcal{R}_T^W \circ \mathcal{R}es_{T}^W \colon \mathfrak{M}_{G|T} \la S_W\text{-}mod,\qquad [X]\mapsto \hh_T(X),\; \alpha\mapsto \alpha_*
\]
is faithful. In particular, there is the inclusion of rings
\[\mathcal{E}_{G|T}(X) \hookrightarrow \End_{S_W}(\hh_T(X)).\]
\end{cor}

\paragraph{\it Parabolic Demazure modules}
We recall the algebraic construction of the cohomology $\hh_T(G/P)$ from \cite{CZZ2} and the construction
of the $W$-action from \cite{LZZ}.

First, following~\cite[\S3]{CZZ2} we identify the equivariant coefficient ring $S=\hh_T(\pt)$ with the so called formal group algebra 
\[
R[[x_\lambda]]_{\lambda\in T^*}/(x_0,x_{\lambda+\mu}-F(x_\lambda,x_\mu))
\] 
corresponding to the formal group law $F$ of the theory $\hh$. 

At the next step, assuming that $S$ satisfies \cite[Assumption~5.1]{CZZ2} (which basically says that all $x_\alpha$'s are regular in $S$) we localize $S$ at all $x_\alpha$'s (characteristic classes) corresponding to positive roots $\alpha$ of the root system $\Sigma$ of $G$. The resulting localized ring is denoted by $Q$. 

Now consider the respective localized twisted group algebra $Q_W$ of $W$. 
The subalgebra of $Q_W$ generated by the so called Demazure elements $X_\alpha={x_\alpha}^{-1}(1-\delta_\alpha)$, $\alpha\in \Sigma$ 
and elements of $S$ is called the {\it formal affine Demazure algebra} and is denoted by $D$. 

\begin{ex}
By \cite{NPSZ} the algebra $D$ can be identified with the convolution algebra $(\hh_G(G/B\times G/B),\circ)$. 
For the Chow theory  $D$ coincides with the nil (affine) Hecke algebra and for the $K$-theory it gives the $0$-affine Hecke algebra.
\end{ex}

Let $W^P$ denote the set of unique minimal left coset representatives of $W/W_P$.
Consider a free $Q$-module $Q_{W/W_P}$ on the basis $\{\delta_{w}\}$ indexed by $w\in W^P$. 
We denote by $D_P$ the image of $D$ under the canonical projection $p\colon Q_W \to Q_{W/W_P}$, and call it the parabolic Demazure algebra.

We will extensively use two key results concerning $D_P$:
The first (see \cite[Theorem~8.11]{CZZ2}) says that the cohomology ring $\hh_T(G/P)$ 
can be identified with the $S$-dual $D_P^\star=\Hom_S(D_P,S)$. 
The second (see \cite[\S3]{LZZ}) shows that the $W$-action on $\hh_T(G/P)$ is, indeed, obtained by restricting the so called $\odot$-action of $Q_{W}$ on $Q_{W/W_P}^*=\Hom_Q(Q_{W/W_P},Q)$ to the action of $S_W$ on $D_P^\star$. 
Together with Corollary~\ref{cor:inclus} it gives
\begin{cor}\label{cor:mainGP} 
For any parabolic subgroup $P$ of $G$ there is the inclusion of rings
\[
\mathcal{E}_{G|T}(G/P) \hookrightarrow \End_{S_W}(D_P^\star).
\]
\end{cor}
Moreover, in \cite{LZZ} it was shown that the $\odot$-action restricts to the action of $D$ on $D_P^\star$ (recall that $S_W\subset D\subset Q_W$). Hence, any $S_W$-equivariant endomorphism of $D_P^\star$ is $D$-equivariant.
So we can replace the endomorphism ring $\End_{S_W}(D_P^\star)$ of the corollary by
$\End_{D}(D_P^\star)$. 

\begin{dfn}
The $D$-module $D_P^\star$ with respect to the $\odot$-action 
is called the parabolic module.
\end{dfn}

\paragraph{\it Direct sum decompositions}
Recall that if $\mathcal{E}$ is the endomorphism ring of an object $M$ in some pseudoabelian category $\mathcal{A}$ (e.g. $M=[X]$, $\mathcal{E}=\mathcal{E}_{G|T}(X)$ in $\mathcal{A}=\mathfrak{M}_{G|T}$), then there is a 1-1 correspondence between direct sum decompositions of $M$ and 
complete finite systems of pairwise orthogonal idempotents $\{p_i\}_{i\in I}$ in $\mathcal{E}$
\begin{equation}\label{eq:11corr}
M=\bigoplus_{i\in I} \mathrm{coker}(p_i) \quad \longleftrightarrow\quad  \sum_{i\in I} p_i =\mathrm{id},\; p_i\circ p_j=0, i\neq j.
\end{equation}
Moreover, two direct summands $\mathrm{coker}(p_i)$ and $\mathrm{coker}(p_j)$ of $M$ are isomorphic in $\mathcal{A}$ (hence, so are the idempotents $p_i$ and $p_j$ of $\mathcal{E}$)
if and only if there exist $\theta_{ij}\in p_i \circ \mathcal{E} \circ p_j$ and $\theta_{ji}\in p_j\circ \mathcal{E} \circ p_i$ such that $\theta_{ij}\circ \theta_{ji}=p_i$ and $\theta_{ji}\circ \theta_{ij}=p_j$.

We say that two direct sum decompositions of $M$ (resp. two systems of idempotents in $\mathcal{E}$) are isomorphic if they coincide up to isomorphisms of summands. 
Finally, we say that $M$ is indecomposable if there are no non-trivial direct summands of $M$ (resp. if there are no non-trivial idempotents in $\mathcal{E}$).

Translating the inclusion of Corollary~\ref{cor:mainGP} into the language of idempotents and direct sum decompositions we obtain
\begin{cor}\label{cor:deco}
If $D_P^\star$ is indecomposable (as the parabolic $D$-module), then so is
$[G/P]$ in $\mathfrak{M}_{G|T}$ (as the motive).
\end{cor}

\section{$G$-equivariant vs. relative equivariant motives}

In the present section we relate the categories of $G$-equivariant motives and relative equivariant motives of split flags. 
Our main results are Theorem~\ref{prop:maincor} and Corollary~\ref{cor:GT}.

\ 

\paragraph{\it Motives of versal flag varieties} 
Let $E$ be a $G$-torsor over a field $k$.
Consider the usual (non-equivariant) category of $\hh$-motives $\mathfrak{M}_k$ generated by the motives (and their direct summands) 
of the twisted flag varieties $E/P$'s for all parabolic subgroups $P$'s (see e.g. \cite[\S3]{GV}). Observe that the morphisms in $\mathfrak{M}_k$ are given by graded correspondences in
$\hh(X\times Y)$.

Following \cite[\S4]{NPSZ} 
consider the composite  of $P\times P$-equivariant maps 
\[
E\times E\stackrel{\simeq}\la E\times G\la G,
\]
where first map is the isomorphism by the definition of a torsor, the second map is the projection on the second factor,
the action of $P\times P$ on the $E\times E$ is given by $(e_1,e_2)\cdot(p_1,p_2)=(e_1p_1,e_2p_2)$ and on $G$ is given by $g\cdot(p_1,p_2)=p_1^{-1}gp_2$.
The induced pullback
\[
\gamma\colon \hh_G(G/P\times G/P)\la\hh_{P\times P}(E\times E)=\hh(E/P\times E/P)
\]
is the ring homomorphism 
$\gamma\colon \mathcal{E}_{G}(G/P) \to \mathcal{E}(E/P)$
from endomorphisms of the $G$-equivariant motive $[G/P]$ to endomorphisms of the (non-equivariant) motive $[E/P]$.

Consider the following important case of a torsor $E$ and variety $E/P$:

Suppose $G$ is a split reductive group defined over a field $k'$ of characteristic 0.
Recall that a $G$-torsor $E$ is called versal (or generic) if it is the generic fiber of the quotient map $GL_N \to GL_N/G$ for an embedding of $G$ into $GL_N$ for some $N\ge 1$, i.e.\  $E$ is a $G$-torsor over the function field $k=k'(GL_N/G)$ (see e.g. \cite{Ka17}). 
Given such $E$, the respective twisted flag variety $E/P$ over $k$ is called the versal flag variety.
Informally speaking, the versal torsor $E$ (resp. $E/P$) can be viewed as the `most twisted' form of $G$ (resp. $G/P$).

\paragraph{\it Lifting of idempotents}
We recall several facts concerning the lifting of idempotents following \cite[\S2]{PSZ}.

Suppose $\gamma\colon \mathcal{E}' \to \mathcal{E}$ 
is a ring homomorphism.
We say that $\gamma$ lifts idempotents and isomorphisms between them if for any 
complete finite system of pairwise orthogonal idempotents $\{p_i\}_{i\in I}$ in $\mathcal{E}$ there exists
a complete system of pairwise idempotents $\{q_i\}_{i\in I}$ in $\mathcal{E}'$ such that $\gamma(q_i)=p_i$ and 
for any $\theta_{ij}$, $\theta_{ji}$ defining an isomorphism between $p_i$ and $p_j$ 
there exist $\theta_{ij}'$, $\theta_{ji}'$ which define an isomorphism between $q_i$ and $q_j$ with 
$\gamma(\theta_{ij}')=\theta_{ij}$, $\gamma(\theta_{ji}')=\theta_{ji}$.
Such $\gamma$ then induces the 1-1 correspondence between isomorphisms classes of
systems of idempotents in $\mathcal{E}$ and $\mathcal{E}'$.

We may additionally assume that both rings $\mathcal{E}$ and $\mathcal{E}'$ are graded, 
e.g.\  for motives. 
In this case $\gamma$ has to be a homomorphism of graded rings, 
all idempotents have degree $0$ and all isomorphisms are homogeneous (see \cite[Definition~2.3]{PSZ}).
In view of \eqref{eq:11corr} such $\gamma$ 
will induce the 1-1 correspondence between isomorphism classes
of direct sum decompositions of the respective motives.

We are now ready to state our main result.

\begin{thm}\label{prop:maincor} Suppose $G$ is a split reductive algebraic group over a field $k$ of characteristic 0. 
Suppose $\hh$ is a free theory. 
Consider the associated $G$-equivariant theory $\hh_G$ and let $\mathfrak{M}_G$ be the respective category of $G$-equivariant motives.
Suppose $E$ is a versal $G$-torsor over $k$.

The pull-back $\gamma$ 
induces a 1-1 correspondence (up to isomorphisms) between direct sum decompositions of the motive $[G/P]$ in  $\mathfrak{M}_G$ and direct sum decompositions 
 of the motive $[E/P]$ in $\mathfrak{M}_k$.
\end{thm}

\begin{proof} Consider the flag variety $G/P$.
Let $\{U_i\}_i$ be the sequence of $G$-representations from the Borel construction for $\hh_G$. 
Then $\hh_{G}(G/P\times G/P)=\varprojlim_i\; \hh((U_i^2\times G)/P^2)$, where the latter quotient is defined on points $(u_1,u_2)\in U_i^2$, $g\in G$, $(p_1,p_2)\in P^2$ as
\[
(U_i^2\times G)/P^2=(U_i^2\times G)/(u_1,u_2,g)\sim (u_1p_1,u_2p_2,p_1^{-1}g p_2).
\]

Let $V\supset E$ be the ambient $G$-representation, i.e.\ $E$ is $G$-invariant open in $V$.  
The map $\gamma$ is the limit of
surjective pullbacks (cf.~the proof of \cite[Lemma~7.9]{NPSZ})
\begin{align*}
\gamma_i\colon\hh((U_i^2\times G)/P^2)\stackrel{\simeq}\la & \hh((U_i^2\times V\times G)/P^2)\la \\
& \hh((U_i^2\times E\times G)/P^2) \stackrel{\simeq}\la\hh((U_i^2\times E\times E)/P^2).
\end{align*}

We now slightly modify the proof of~\cite[Lemma~3.2]{VZ}. 
By the localization sequence for $\hh$ each element $\phi$ in the kernel of $\gamma_i$ lies in the image of 
\[\hh((U_i^2\times Z\times G)/P^2) \stackrel{\iota_*}\la \hh((U_i^2\times V \times G)/P^2),\] 
where $Z=V\setminus E \stackrel{\iota}\hookrightarrow V$ is the closed complement of $E$ in $V$.
Consider the quotient $(U_i^{d+1}\times G^d)/P^{d+1}$ defined as $U_i^{d+1}\times G^d$ modulo 
\[
(u_1,\ldots,u_{d+1},g_1,\ldots,g_d)\sim (u_1p_1,\ldots,u_{d+1}p_{d+1},p_1^{-1}g_1p_2,\ldots, p_1^{-1}g_dp_{d+1})
\]
for $(u_1,\ldots,u_{d+1})\in U_i^{d+1}$, $(g_1,\ldots,g_d)\in G^d$ and $(p_1,\ldots,p_{d+1})\in P^{d+1}$.
Consider the maps
\[\pi_{j,j'}\colon (U_i^{d+1}\times G^d)/P^{d+1}\la (U_i^2\times G)/P^2,\quad 1\le j<j'\le d+1\]
which descend from the standard projection $U_i^{d+1}\to U_i^2$ on the $j$ and $j'$ components 
and the map $G^d\to G$ defined on points by $(g_1,\ldots,g_d)\mapsto g_{j-1}^{-1}g_{j'-1}$ for $j>1$ and
$(g_1,\ldots, g_d) \mapsto g_{j'-1}$ for $j=1$.
Then the $d$-fold correspondence product on $\hh((U_i^2\times G)/P^2)$ is given by 
\[
\phi \circ \ldots \circ \phi=(\pi_{1,d+1})_*(\prod_{j=1}^d \pi_{j,j+1}^*(\phi)).
\]
Since for each $j$ there is the Cartesian square
\[
\xymatrix{
(U_i^{d+1} \times Z\times G^d)/{P^{d+1}}\ar[d] \ar[r]^\iota & (U_i^{d+1} \times V\times G^d)/{P^{d+1}}\ar[d]^{\pi_{j,j+1}}\\
(U_i^{2} \times Z\times G)/{P^{2}}\ar[r] & (U_i^{2} \times V\times G)/{P^{2}},
}
\]
for every $\phi \in \hh((U_i^{2} \times V\times G)/{P^{2}})$ 
supported on $(U_i^{2} \times Z\times G)/{P^{2}}$ the element $\pi_{j,j+1}^*(\phi)$ 
is supported on $(U_i^{d+1} \times Z\times G^d)/{P^{d+1}}$. 
By~Lemma~\ref{lem:hintersect} applied to 
\[
p_{j,j+1}\colon Y=(U_i^{d+1} \times V\times G^d)/{P^{d+1}}\la X=(U_i^{2} \times V\times G)/{P^{2}}\] 
we obtain that for $d>\dim(V)/\codim(Z)$ we have
$\prod_{j=1}^d \im (\iota_*)=0$, and, hence, $\phi^{\circ d}=0$. 
Therefore, $\gamma$ is a limit of surjective maps with nilpotent kernels. 

Now by~\cite[Prop.~6.2.1]{Ne16} each inclusion $U_i\to U_{i+1}$ induces a surjective homomorphism $\ker(\gamma_{i+1})\to\ker(\gamma_{i})$. 
So by~\cite[Lemma 4.3.4]{Ne16} which is the limit-generalization~of~\cite[Prop.~2.6]{PSZ}, the map $\gamma$ lifts idempotents and isomorphisms between idempotents
in the sense of \cite[Def.~2.3]{PSZ}. Hence, it induces the 1-1 correspondence between isomorphism classes of direct sum decompositions of $[G/P]$ in $\mathfrak{M}_G$
and $[E/P]$ in $\mathfrak{M}_k$.
\end{proof}

\begin{cor}\label{cor:GT} Under the hypothesis of the theorem
there is a 1-1 correspondence between direct sum decompositions of the motive $[G/P]$ in  $\mathfrak{M}_G$ 
and in $\mathfrak{M}_{G|T}$.
\end{cor}

\begin{proof}
In the commutative diagram
\[
\xymatrix{
\hh_G(G/P\times_k G/P) \ar[r]^\gamma \ar[d]_{\res^G_T} &  \hh(E/P\times_k E/P)\ar[d]^{\phi} \\
\hh_T(G/P\times_k G/P)  \ar[r]^{\res^T_{\bar k}} &  \hh(G/P\times_{\bar k} G/P)
}
\]
the map $\phi$ has nilpotent kernel by the Rost Nilpotence Principle for free theories (see \cite[Corollary~4.4]{GV}). 
By the theorem, $\gamma$ is a limit of surjective maps with nilpotent kernels, hence, so is the restriction $\res^G_T$. 
The result then follows again by the lifting of idempotents (see e.g.~\cite[Lemma 4.3.4]{Ne16}). 
\end{proof}

Summarizing the above arguments 
we obtain the following chain of 1-1 correspondences between direct sum decompositions of the respective motives:
\[
[\text{versal }E/P] \in \mathfrak{M}_k \longleftrightarrow [G/P] \in \mathfrak{M}_G \longleftrightarrow [G/P] \in \mathfrak{M}_{G|T}.
\]
Combining this with Corollary~\ref{cor:deco} we obtain
\begin{cor}\label{cor:mainindec}
If the parabolic $D$-module $D_P^\star$ is indecomposable, then so is the motive   (in $\mathfrak{M}_k$) of the versal flag variety $E/P$.
\end{cor}


\section{Localized endomorphisms}\label{sec:localized}

In this section we study the endomorphism ring $\End_{D}(D_P^\star)$ using the localization techniques of \cite{CZZ1,CZZ2} and the results
from \cite[\S3]{LZZ}. Our main results are Corollary~\ref{cor:doublec} and Lemma~\ref{lem:condd}.

\

\paragraph{\it Endomorphisms of localized modules}
Consider the $Q$-dual $Q_{W/W_P}^*$ with the basis $\{f_w\}$, $w\in W^P$ dual to $\{\delta_w\}$. 
Following \cite[\S3]{LZZ} the $\odot$-action of $Q_W$ on $Q_{W/W_P}$ is defined by
\[
q\delta_w\odot pf_v=qw(p)f_{\overline{wv}},\quad q,p\in Q,\;\; v\in W^P,
\]
where $\overline{wv}$ denotes the minimal coset representative of $wvW_P$. Consider the endomorphism ring 
$E_Q=\End_{Q_W}(Q_{W/W_P}^*)$ of $Q_W$-modules with respect to the $\odot$-action. 

Since the $Q_W$-module $Q_{W/W_P}^*$ is generated by $f_1$,
any  $\phi\in E_Q$ is uniquely determined by it value at $f_1$ that is 
\begin{equation}\label{eq:pres}
\phi(f_1)=\sum_{w\in W^P}c_{w}f_{w}, \; c_{w}\in Q.\end{equation}
Moreover, we have
\begin{lem}\label{lem:conded} $\phi \in E_Q$ $\Longleftrightarrow$ 
$v(c_{w})=c_{\overline{vw}}$,  $\forall v\in W_P$, $\forall w\in W^P$.
\end{lem}

\begin{proof}
$\Longrightarrow:$ Since $\delta_v\odot \phi(f_{1}) = \phi(\delta_v\odot f_{1})=\phi(f_{1})$ for all $v\in W_P$, we obtain
\[
\sum_{w\in W^P} v(c_{w})f_{\overline{vw}}=\sum_{w\in W^P}c_{w} f_{w}\quad \text{for all }v\in W_P.
\]

$\Longleftarrow:$ 
For each $w'\in W^P$ we set  $\phi(f_{w'}):=\delta_{w'}\odot \phi(f_1)$. 
It is enough to show that $\delta_w\odot \phi(f_{w'})=\phi(\delta_w\odot f_{w'})$ for all $w,w'\in W^P$ or, equivalently,
\[
\delta_{ww'}\odot \phi(f_1)=\phi(f_{\overline{ww'}})=\delta_{\overline{ww'}}\odot \phi(f_1).
\]
But $ww'=\overline{ww'}v$, $v\in W_P$ and $\delta_v\odot \phi(f_1)=\phi(f_1)$ by the assumption. 
\end{proof}

Let ${}^PW^P$ denote the set of minimal double $W_P$-coset representatives. 
Since each $w\in W^P$ can be written uniquely as $w=vu$, where $v\in W_P$ and $u\in {}^PW^P$,
we can rewrite \eqref{eq:pres} as
$\phi(f_1)=\sum_{w\in W^P,\, w=vu} v(c_u)f_w$ and we obtain
\begin{cor}\label{cor:doublec} $\phi\in E_Q$ is uniquely determined by the coefficients $c_{u}\in Q^{W_P\cap uW_Pu^{-1}}$, $u\in {}^P W^P$. \end{cor}

\begin{ex}
Let $G$ be of Dynkin type $\Atype_n$ and let $P$ be of type $\Atype_{n-1}$, i.e.\ $G/P=\PP^{n}$.
We have $W=\langle s_1,\ldots,s_n\rangle$, $W_P=\langle s_2,\ldots,s_{n}\rangle $ and $W^P=\{1,v_1,v_2,\ldots,v_n\}$,
where $s_i$ denotes the $i$-th simple reflection and $v_i=s_is_{i-1}\ldots s_1$.

For any $\phi\in \End_{Q_W}(Q_{W/W_P}^*)$ presentation \eqref{eq:pres} can be written as \[
\phi(f_{\bar 1})=c_0 f_{1}+c_1 f_{v_1}+\ldots +c_n f_{v_n},\text{ where }c_i\in Q.
\]
Since there are only two double cosets $W_P$ and $W_Ps_1W_P$, $\phi$ is determined by two coefficients $c_0\in Q^{\langle s_2,\ldots,s_n\rangle}$ and $c_1\in Q^{\langle s_3,\ldots, s_n\rangle}$ with 
$c_j=\tfrac{v_j}{v_1}(c_1)$, $1\le j\le n$.
\end{ex}

\paragraph{\it Localized endomorphisms} The localization $D \subset Q_W$ induces the inclusion
\[
E_D=\End_{D}(D_P^\star)\hookrightarrow E_Q=\End_{Q_W}(Q_{W/W_P}^*).
\]
We investigate its image.

According to \cite[\S3]{LZZ} $D_P^\star$ is a $D$-module generated 
by the class of a point \[[\pt]=x_{\Pi/P}f_{1} \in S_{W/W_P}^\star=\Hom_{S}(S_W,S),\] where $x_{\Pi/P}=x_\Pi/x_P$, $x_P=\prod_{\alpha\in \Sigma_P^{-}}x_\alpha$ and $x_\Pi=\prod_{\alpha\in \Sigma^-}x_\alpha$, $\Sigma_P^-$ is the set of all negative roots of the root subsystem for $P$.

Therefore, any $\phi\in E_D$ is uniquely determined by its value on $[\pt]$. 
On the other hand, $\phi([\pt])$ belongs to  $D_P^\star$ as an element of 
$S_{W/W_P}^\star \subset Q_{W/W_P}^*$ if and only if it satisfies the criteria of \cite[Thm.~11.9]{CZZ1}.
Combining these together we obtain 
\begin{lem}\label{lem:condd}
An endomorphism $\phi\in E_Q$ comes from $E_D$ if and only if its coefficients $c_{w}\in Q$ satisfy
\[
x_{\Pi/P}c_{w} \in S\text{ and }x_{w(\alpha)}\mid x_{\Pi/P}(c_{w} - c_{\overline{s_{w(\alpha)}w}})\text{ for all }\alpha \notin \Sigma_P.
\]
\end{lem}

The localization $S_W \subset Q_W$ 
also induces the inclusion 
\[
E_S=\End_{S_W}(S_{W/W_P}^\star) \hookrightarrow E_Q=\End_{Q_W}(Q_{W/W_P}^*)
\]
where $\phi\in E_Q$ comes from $E_S$ if and only if all the coefficients $c_w$ of $\phi$ are in~$S$. 
From now on we identify $E_D$ and $E_S$ with their images in $E_Q$.
We claim that 
\begin{lem}\label{lem:restrend} $E_S\subset E_D$ in $E_Q$.
\end{lem}

\begin{proof}
Consider an endomorphism $\phi \in E_S$ with $\phi(f_{1})=\sum c_{w} f_{w}$, $c_w\in S$.
By Lemma~\ref{lem:condd} it is enough to show that
\[
x_{w(\alpha)}\mid x_{\Pi/P}(c_w - c_{\overline{s_{w(\alpha)}w}})\text{ for all }\alpha \notin \Sigma_P.
\]
If $w(\alpha)\notin \Sigma_P$, then $x_{w(\alpha)} \mid x_{\Pi/P}$, hence we may 
assume $w(\alpha) \in \Sigma_P$.
By Lemma~\ref{lem:conded}, $c_{\overline{s_{w(\alpha)}w}}=s_{w(\alpha)}(c_w)$. Then $x_{w(\alpha)} \mid (c_w - s_{w(\alpha)}(c_w))$ by \cite[Corollary~3.4]{CPZ}.
\end{proof}
Finally, observe that
the endomorphism $\phi$ which maps $[\pt]$ to $\sum_{w\in W^P} f_w\in D_P^\star$ 
clearly belongs to $E_D$ but not to $E_S$ in general.


\section{Endomorphisms in terms of the Schubert basis}

In the present section we concentrate on the study of the Chow theory $\hh=\CH$ 
(here $R=\Z$ or $2$ is not a zero-divisor in $R$). In this case $S$ is the polynomial ring 
in the basis of characters $T^*$, $X_i=X_{\alpha_i}=(1-\delta_{s_i})/\alpha_i$ corresponds to the classical divided difference operator $\Delta_i$ for  the simple root $\alpha_i$ and $D$ is the nil (affine) Hecke algebra.

\ 

\paragraph{\it Endomorphisms of $D$-modules}
Recall that there is the free $S$-basis of $\hh_T(G/P)$ given by the so called Schubert classes $\xi_w$, $w\in W^P$.
The latter are the classes of the Schubert varieties of dimensions $l(w)$ (the length of $w$), 
the closures of the orbits $BwP/P$ in $G/P$.
We have the following well-known formula for the $\odot$-action by divided difference operators
on the Schubert classes (see e.g. \cite[Example~4.8]{LZZ})
\[
X_j \odot \xi_v=
\begin{cases}
\xi_{s_jv} & \text{if } l(s_jv)\ge l(v) \text{ and }s_jv\in W^P \\
0 & \text{otherwise.}
\end{cases}
\]

We now describe endomorphisms of $D_P^\star$ with respect to the Schubert basis $\xi_w$.

Let $\phi\in E_D=\End_D(D_P^\star)$.
As an endomorphism of free $S$-modules $\phi$ is uniquely determined by its values on the classes $\xi_w$, $w\in W^P$
that is 
\[\phi(\xi_w)=\sum_{v}a_{v,w}\xi_v,\] where $(a_{w,v})$ is the corresponding generalized matrix of coefficients in $S$ 
(here the coefficients are indexed by the product of left Bruhat posets in $v$ and $w$).

Since $D$ is generated by $X_j$'s and elements of $S$, $\phi \in E_D$  if and only if 
\[
X_j\odot \phi(\xi_w)=\phi(X_j\odot \xi_w)\text{ for all }j\text{ and }w\in W^P.
\]

Using the properties of the $\odot$-action of \cite[\S3]{LZZ} we compute the left hand side as follows (for $v,w\in W^P$)
\begin{align*}
X_j\odot \sum_v a_{v,w}\xi_v &=\sum_{v} (s_j(a_{v,w})X_j+\Delta_j(a_{v,w}))\odot \xi_v \\
&=\sum_{s_jv\in W^P,\; l(s_jv)\ge l(v)} s_j(a_{v,w})\xi_{s_jv}+\sum_{v}  \Delta_j(a_{v,w})\xi_v\\
&=\sum_{v'\in W^P,\; l(s_jv')\le l(v')}s_j(a_{s_jv',w})\xi_{v'}+\sum_v \Delta_j(a_{v,w})\xi_v.
\end{align*}
As for the right hand side we obtain
\[
\phi(X_j\odot \xi_w)=
\begin{cases}
\sum_v a_{v,s_jw}\xi_v & \text{ if } l(s_jw)\ge l(w)\text{ and }s_jw\in W^P, \\
0  & \text{otherwise}.
\end{cases}
\]
Combining, we obtain the following recursive formulas for the coefficients $a_{v,w}$:
\begin{lem}\label{lem:recursive}
If $l(s_jw)\ge l(w)$ and $s_jw\in W^P$, then
\[
a_{v,s_jw}=
\begin{cases}
s_j(a_{s_jv,w})+\Delta_j(a_{v,w}) & \text{ if }l(s_jv)\le l(v), \\
\Delta_j(a_{v,w}) & \text{ otherwise}.
\end{cases}
\] 

If $l(s_jv)\le l(v)$, then (observe that $s_j\Delta_j=\Delta_j$ in the Chow theory)
\[
a_{s_jv,w}=\begin{cases}
s_j(a_{v,s_jw})-\Delta_j(a_{v,w}) & \text{ if }l(s_jw)\ge l(w)\text{ and }s_jw\in W^P, \\
-\Delta_j(a_{v,w}) & \text{otherwise}.
\end{cases}
\]
\end{lem}

We then obtain the following analogues of Lemma~\ref{lem:conded} and Corollary~\ref{cor:doublec}:

\begin{cor}\label{cor:res} 
An endomorphism $\phi\in E_D$ is uniquely determined by the coefficients 
$a_v=a_{v,1} \in S^{W_P}$, $v\in W^P$ such that $l(s_jv)\ge l(v)$ for all $s_j\in W_P$.
\end{cor}

\begin{proof}
It follows immediately by the recurrent formulas 
that $\phi$ is uniquely determined by its value at $[\pt]=\xi_1$, i.e.\  by the coefficients $a_v$.
Comparing $X_j\odot \phi(\xi_1)$ and $\phi(X_j\odot \xi_1)$ we obtain that for all $s_j\in W_P$
\[
\Delta_j(a_v)=
\begin{cases}
-a_{s_jv} & \text{ if } l(s_jv)\le l(v) \\
0 & \text{ otherwise.}
\end{cases} \qedhere
\]
\end{proof}

\paragraph{\it  Homogeneous endomorphisms}
Recall that the endomorphism ring 
of the relative equivariant motive 
consists of graded correspondences $\alpha \in \CH_T^n(G/P \times G/P)$, where $n=\dim G/P$. The respective realization map
$\phi=\alpha_*$ then preserves the dimensions of cycles and, hence, the dimensions $l(w)$ of the respective Schubert
classes $\xi_w$.
In other words, the inclusion of Corollary~\ref{cor:mainGP} 
\[
\mathcal{E}_{G|T}([G/P]) \hookrightarrow E_D=\End_{D}(D_P^\star),\qquad \alpha \mapsto \alpha_*,
\]
factors through the subring $E_D^{(0)}$ of endomorphisms $\phi$ 
where each coefficient $a_{v,w}$ is a homogeneous polynomial of degree $\dim \xi_v - \dim \xi_w= l(v)-l(w)$. 
We call such $\phi$ a homogeneous endomorphism.

If $\phi\in E_D^{(0)}$, then $\phi([\pt])=\sum_{v\in W^P} a_v \xi_v$, where $a_v=a_{v,1}\in S$ is homogeneous of degree $l(v)$. Since $\xi_v=\sum_{w\in W^P} c_{w,v} f_w$ where all $c_{w,v}$s are homogeneous of degree $n-l(v)$, we get
$\phi(f_1)=\sum_{v\in W^P} c_w f_w$ where each coefficient
$c_w=\tfrac{1}{x_{\Pi/P}}\sum_v a_vc_{w,v}$
is the quotient of homogeneous polynomials of degree $n$. 

Now $\phi$ comes from $E_S$ of Lemma~\ref{lem:restrend} if and only if 
each $c_w\in S$ is a polynomial of degree 0, i.e.\  each $c_w\in R$. 
Since $R_W=R[W]$, $R_{W/W_P}^\star=R[W/W_P]$ 
and the $\odot$-action restricted to $S=R$ coincides with the left action of $W$ on $W/W_P$, we obtain
\begin{cor}\label{cor:integr} 
We have $E_S\cap E_D^{(0)}=\End_{R[W]}(R[W/W_P])$. 
In particular, if $D_P^\star$ is indecomposable (as a $D$-module with $R=\Z$), then so is
the induced integral representation $\mathrm{Ind}^W_{W_P}\mathbf{1}$ of the Weyl group $W$.
\end{cor}

\paragraph{\it Idempotent matrices}
Observe that $a_{v,w}=0$ if $l(v)-l(w)<0$. So the generalized matrix of coefficients $(a_{v,w})$  of a homogeneous endomorphism $\phi$ is always a lower triangular matrix.

\begin{lem}\label{lem:idlem}
The generalized matrix of coefficients $(a_{v,w})$
of an idempotent $\phi$ is a lower triangular idempotent matrix. In particular, its coefficients satisfy
\[
\sum_{u,\; l(u)\le l(v)} a_u a_{v,u}=a_v,\quad  v,u\in W^P
\] 
and its diagonal is the direct sum of idempotent matrices 
\[\bigoplus_{d=0}^n (a_{v,w})_{\{(v,w) \mid l(v)=l(w)=d\}}, \quad n=\dim G/P.\]
\end{lem}

\begin{proof}
Any idempotent $\phi$ is homogeneous, so $(a_{v,w})$ is triangular and idempotent.
By definition, $\phi(\phi([\pt]))=\phi([\pt])$ which leads to 
\[
\sum_v a_v\xi_v= \phi(\sum_u a_u\xi_u)=\sum_u a_u\phi(\xi_u)=\sum_{u,v} a_ua_{v,u}\xi_v
\]
and, hence, to the formula for the coefficients. The latter statement follows from the properties of triangular idempotent matrices.
\end{proof}

\section{Applications to Chow motives}

The purpose of the present section is to demonstrate how the recursive formulas of Lemma~\ref{lem:recursive}
can be applied to show indecomposability of certain $D$-modules $D_P^\star$ 
and, hence, of Chow motives of versal flag varieties $E/P$. 

\subsection{Projective spaces}
We fix a simple split group $G$ of type $\Atype_n$ and its parabolic subgroup $P$ of type $\Atype_{n-1}$. 
The respective $G/P$ is isomorphic to the projective space $\PP^n$.
We consider the adjoint form $G=PGL_{n+1}$ so that $T^*$ is the root lattice with the basis 
given by simple roots $\{\alpha_1,\ldots,\alpha_n\}$. Set $R=\Z$.
The algebra $D$ is then the usual nil-Hecke algebra over the polynomial ring 
$S=\Z[T^*]=\Z[\alpha_1,\ldots,\alpha_n]$.

Consider an idempotent $\phi\in E_D^{(0)}$ and the associated generalized matrix of coefficients $(a_{v,w})$.

We have $W=\langle s_1,\ldots,s_n\rangle$ where $s_i$ is the simple reflection corresponding to $\alpha_i$, $W_P=\langle s_2,s_3,\ldots,s_{n}\rangle$ and $W^P=\{1,v_1,v_2,\ldots,v_n\}$,
where $v_i=s_is_{i-1}\ldots s_1$.
Set $c_{i,j}:=a_{v_i,v_j}$,  $c_{i,0}:=a_{v_i,1}$, $c_{0,j}=a_{1,v_j}$ for $i,j\ge 1$ and $c_{0,0}=a_{1,1}$. 
So $(c_{i,j})=(a_{v,w})$ is a lower-triangular idempotent $(n+1)\times (n+1)$-matrix with polynomial coefficients of degrees 
$\deg c_{i,j}=i-j$.

The recursive formulas of Lemma~\ref{lem:recursive} turn into:
\begin{align*}
c_{i,i}&=s_i(c_{i-1,i-1})+\Delta_i(c_{i,i-1})  & \text{if} \; i\ge 1,\\
c_{i,j}&=\Delta_j(c_{i,j-1})   & \text{if}\;i\neq j,\; j\ge 1,\\
c_{i,0}&=(-1)^{n-i}\Delta_{i+1,..,n}(c_{n,0})  & \text{for}\; n> i\ge 1.
\end{align*}
By the formulas we obtain:
\begin{align*}
c_{i,j} &=\Delta_{j,..,1}(c_{i,0})=(-1)^{n-i}\Delta_{j,..,1,i+1,..,n}(c_{n,0}), &  \text{ if } i>j\ge 1 \text{ (under the diagonal),}\\
c_{i,i} &=s_i(c_{i-1,i-1})+(-1)^{n-i}\Delta_{i,..,1,i+1,..,n}(c_{n,0}) & \text{ (on the diagonal)}.
\end{align*}
In other words, we have the following diagram of operators
{\small \[
\xymatrix{
c_{0,0} \ar@{.>}[rd]^{s_1} & & & \\
c_{1,0} \ar[r]^{\scriptscriptstyle\Delta_1}& c_{1,1} \ar@{.>}[rd]^{s_2} & & \\
c_{2,0} \ar@{}[d]|{\vdots} \ar[r]^{\scriptscriptstyle\Delta_1}\ar[u]^{-\scriptscriptstyle\Delta_{2}} & c_{2,1} \ar@{}[d]|{\vdots}  \ar[r]^{\scriptscriptstyle\Delta_2}& c_{2,2} \ar@{}[dr]|{\ddots} &\\
c_{n-2,0} \ar[r]^{\scriptscriptstyle\Delta_1} & \cdots & & c_{n-2,n-2}  \ar@{.>}[rd]^{s_{n-1}} & & \\
c_{n-1,0} \ar[r]^{\scriptscriptstyle\Delta_1} \ar[u]^-{-\scriptscriptstyle\Delta_{n-1}} & \cdots & & \cdots \ar[r]^-{\scriptscriptstyle\Delta_{n-1}}& c_{n-1,n-1}  \ar@{.>}[rd]^{s_n} & \\
c_{n,0} \ar[r]^{\scriptscriptstyle\Delta_1}\ar[u]^-{-\scriptscriptstyle\Delta_n} & \cdots & & \cdots \ar[r]^-{\scriptscriptstyle\Delta_{n-1}} & c_{n,n-1} \ar[r]^-{\scriptscriptstyle\Delta_n}& c_{n,n} 
}
\]}
Using relations in $D$ for the divided difference operators $\Delta_j$ 
and the fact that $\Delta_j(c_{n,0})=0$ for all $j\neq 1,n$ (Corollary~\ref{cor:res}) we get
\[
\Delta_k\circ \Delta_{i-1,..,1,i+1,..,n}(c_{n,0})=0,\quad \text{ for all }k\neq i.
\]
So $c_{i,i-1}\in S^{W_i}$, where $W_i$ is generated by all simple reflections except the $i$-th one.
Since $c_{i,i-1}$ has degree 1, we can express it as
$
c_{i,i-1}=b_1\alpha_1+\ldots+ b_n\alpha_n$, $b_i\in \Z$.
Then $c_{i,i-1}\in S^{W_i}$ is equivalent to
\[
b_2=2b_1,b_3=3b_1,\ldots, b_i= i b_1=(n+1-i)b_n,\ldots,b_{n-2}=3b_n,b_{n-1}=2b_n.
\]

Assume $n+1=p^r$ for some prime $p$ and $r\ge 1$. 
Then 
\[
p\mid \Delta_i(c_{i,i-1})=2b_i-b_{i-1}-b_{i+1}=b_1+b_n.
\] 

Since $\phi$ is an idempotent, all diagonal elements $c_{i,i}$ are idempotent as well by Lemma~\ref{lem:idlem}, i.e.\ $c_{i,j}=0,1$.
The recursive formulas and the fact that $p\mid \Delta_i(c_{i,i-1})$ then imply that $c_{i,i}=c_{i-1,i-1}$ for all $i$, i.e.\ that
there are no nontrivial idempotents in $E_D$.
In other words, we obtain 
\begin{prop} Let $G=PGL_{n+1}$, 
where $n=p^r-1$ for some $r\ge 1$ and a prime $p$. 
Let $P$ be the standard parabolic subgroup of type $\Atype_{n-1}$. 

Then the $D$-module $D_P^\star$ is indecomposable over $\Z$ (over $\Z/p\Z$).
\end{prop}

\begin{rem}
In view of Corollary~\ref{cor:mainindec} this fact implies 
the celebrated result by Karpenko \cite{Ka96} on indecomposability of the integral motive of the Severi-Brauer variety of a generic algebra.
On the other side according to Corollary~\ref{cor:integr} 
it also implies that the induced integral (mod $p$) representation $\mathrm{Ind}^{S_{n+1}}_{S_n} \mathbf{1}$ of the symmetric group 
is indecomposable.
\end{rem}

\subsection{Klein quadrics}
Let $G$ be a simple split group of type $\Atype_3$ 
with the character lattice $T^*=\langle \alpha_1,\alpha_2,\alpha_3,\omega_2 \rangle$ 
(here $\omega_2$ is the respective fundamental weight), i.e.\  $G=SL_4/\mu_2$. 
Let $P$ be a parabolic of type $\Atype_1\times A_1$, i.e.\  
$G/P=Gr(2,4)$ is a split 4-dimensional smooth projective quadric. 

Consider an idempotent $\phi\in E_D^{(0)}$ and the associated generalized matrix $(a_{v,w})$.

By definition the Weyl group $W=\langle s_1,s_2,s_3\rangle$, $W_P=\langle s_1,s_3\rangle$ 
and the set of minimal coset representatives $W^P$ is given by the Hasse diagram
{\small \[
\xymatrix{
1 \ar[r]^{\scriptscriptstyle s_2\cdot }  & s_2 \ar[d]_{\scriptscriptstyle s_3\cdot }\ar[r]^{\scriptscriptstyle s_1\cdot }  & s_1s_2\ar[d]^{\scriptscriptstyle s_3\cdot } & \\
& s_3s_2 \ar[r]^{\scriptscriptstyle s_1\cdot }& s_1s_3s_2\ar[r]^{\scriptscriptstyle s_2\cdot } & s_2s_1s_3s_2
}
\]}
For simplicity, set $a_{ijk..}=a_{s_is_js_k..}$ and $a_\emptyset=a_1$.

By the recursive formulas of Lemma~\ref{lem:recursive} and of Corollary~\ref{cor:res} we obtain:
\begin{align*}
a_{2,2} &=a_\emptyset+\Delta_{2,1,3}(a_{132}), \\
a_{12,12} &=a_{2,2}-\Delta_{1,2,3}(a_{132})\;\; \text{and}\;\; a_{32,32}=a_{2,2}-\Delta_{3,2,1}(a_{132}), \\
a_{32,12} &=-\Delta_{1,2,1}(a_{132})\;\;\text{and}\;\; a_{12,32}=-\Delta_{3,2,3}(a_{132}), \\
a_{132,132} &=a_{12,12}+\Delta_{3}(\Delta_{1,2}-\Delta_{2,1})(a_{132})=a_{32,32}+\Delta_{1}(\Delta_{3,2}-\Delta_{2,3})(a_{132}), \\
a_{2132,2132} &=a_{132,132}+\Delta_{2,1,3,2}(a_{2132})+\Delta_{2,1,3}s_2(a_{132}).
\end{align*}
which can be also expressed as the following diagram of operators (here we denote $a_{ijk..,lmn..}$ by $a_{ijk..}^{lmn..}$)
{\small 
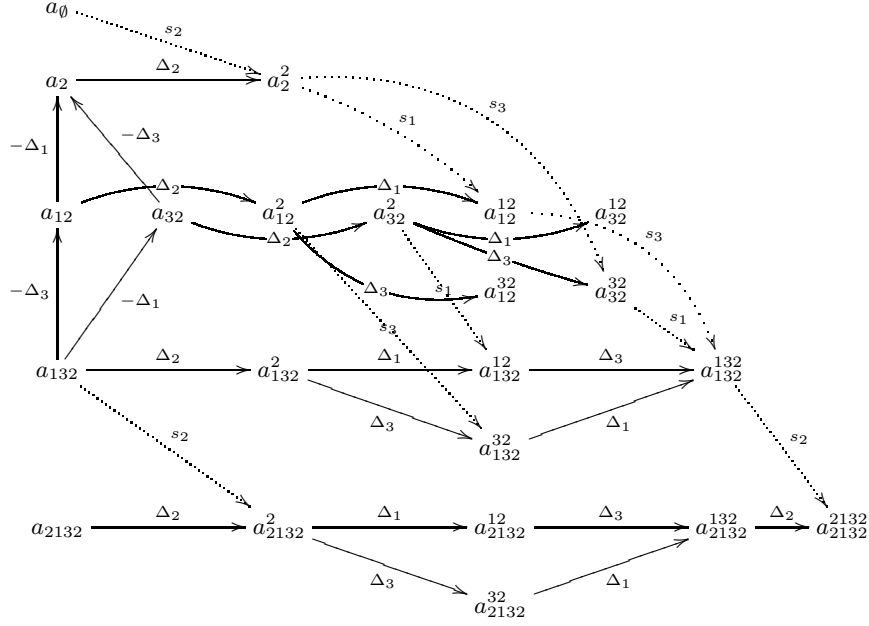
\[
\xymatrix@C=20pt@R=15pt{
a_\emptyset \ar@{.>}[drr]^{s_2} & & & & & & & \\
a_{2} \ar[rr]^{\Delta_2} & & a_{2}^{2} \ar@{.>}@/^/[ddrr]^{s_1} \ar@{.>}@/^30pt/[dddrrr]^{s_3} & & & & & \\
 & & & & & & & \\
a_{12} \ar[uu]^{-\Delta_1} \ar@/^10pt/[rr]|-{\Delta_2} & a_{32} \ar[uul]_{-\Delta_3}  \ar@/_10pt/[rr]|-{\Delta_2} & a_{12}^{2} \ar@/^10pt/[rr]|-{\Delta_1} \ar@/_15pt/[rrd]|-{\Delta_3} \ar@{.>}[rrddd]|-{s_3} & a_{32}^{2} \ar@/_10pt/[rr]|-{\Delta_1} \ar@/_2pt/[rrd]|-{\Delta_3} \ar@{.>}[rdd]|-{s_1}& a_{12}^{12}  \ar@{.>}@/^20pt/[ddrr]^{s_3}  & a_{32}^{12} & & \\
& & & & a_{12}^{32} & a_{32}^{32} \ar@{.>}[dr]^{s_1} & & \\
a_{132}  \ar[rr]^{\Delta_2} \ar@{.>}[ddrr]^{s_2} \ar[uu]^{-\Delta_3} \ar[uur]_{-\Delta_1}  &  &  a_{132}^{2}  \ar[rr]^{\Delta_1}  \ar[rrd]_{\Delta_3} &  &  a_{132}^{12}  \ar[rr]^{\Delta_3} &  & a_{132}^{132}  \ar@{.>}[ddr]^{s_2} &   \\
                 &  &                     & &  a_{132}^{32} \ar[rru]_{\Delta_1}   &  &                        & \\
a_{2132} \ar[rr]^{\Delta_2} &  &  a_{2132}^{2}  \ar[rr]^{\Delta_1}  \ar[rrd]_{\Delta_3}  &  &  a_{2132}^{12} \ar[rr]^{\Delta_3} &  & a_{2132}^{132} \ar[r]^{\Delta_2}  &  a_{2132}^{2132} \\
                 &  &                     & &  a_{2132}^{32} \ar[rru]_{\Delta_1} &  &                        &
}
\]}
\begin{lem}
For any polynomial $g$ of degree 3 we have
\[
\Delta_{3,2,1}(g)\equiv \Delta_{1,2,3}(g)\equiv\Delta_{3,2,3}(g)\equiv \Delta_{1,2,1}(g)\;  \text{ and}
\]
\[
\Delta_{2,1,3}(g)\equiv \Delta_{3,1,2}(g)\equiv \Delta_{1,3,2}(g)\equiv 0 \mod 2.
\]
\end{lem}

\begin{proof}
As for the first chain of congruences, observe that $\Delta_1(g)\equiv \Delta_3(g)\equiv 0$ for any polynomial $g$ which does not contain $\alpha_2$, and the computations are symmetric with respect to 
$\alpha_1$ and $\alpha_3$. So it is enough to check it only on monomials
$\alpha_2^2\alpha_1$ and $\alpha_2\omega_2\alpha_1$. Direct computations $\mod 2$ then give
\begin{align*}
\Delta_{3,2,1}(\alpha_2^2\alpha_1) &\equiv \Delta_{3,2}(\alpha_1^2)\equiv \Delta_3(\alpha_2)\equiv 1,\text{ and } \\
\Delta_{1,2,3}(\alpha_2^2\alpha_1) &\equiv \Delta_{1,2}(\alpha_1\alpha_3)\equiv \Delta_1(\alpha_1+\alpha_2+\alpha_3)\equiv 1;\\
\Delta_{3,2,1}(\alpha_2\omega_2\alpha_1) &\equiv \Delta_{3,2}(\omega_2\alpha_1)\equiv \Delta_3(\alpha_1+\alpha_2-\omega_2)\equiv 1,\text{ and } \\
\Delta_{1,2,3}(\alpha_2\omega_2\alpha_1) &\equiv \Delta_{1,2}(\omega_2\alpha_1)\equiv \Delta_1(\alpha_1+\alpha_2-\omega_2)\equiv 1.
\end{align*}
Similarly, we get $\Delta_{3,2,3}(g)\equiv \Delta_{1,2,1}(g)$ and
$\Delta_{3,2,3}(g)\equiv \Delta_{1,2,3}(g)$.

As for the second chain of congruences, it is enough to verify that $\Delta_{1,3}(h)\equiv 0$ for any quadratic $h$ and $\Delta_{2,1,3}(\alpha_2^3)\equiv 0$.
Indeed, for quadratic $h$ it reduces to $\Delta_{1,3}(\alpha_2^2)\equiv \Delta_1(\alpha_3)\equiv 0$ and 
\[
\Delta_{2,1,3}(\alpha_2^3) \equiv \Delta_{2,1}(\alpha_2^2+\alpha_2\alpha_3+\alpha_3^2)\equiv  \Delta_2(\alpha_1+\alpha_3)\equiv 0. \qedhere
\]
\end{proof}

Now by Lemma~\ref{lem:idlem} since $\phi$ is an idempotent, all the diagonal entries of the generalized matrix $(a_{v,w})$ are idempotents as well. In particular, the matrix 
\[M=\begin{pmatrix} a_{12,12} & a_{32,12} \\ a_{12,32} & a_{32,32}\end{pmatrix}\] is an idempotent matrix.
Since $a_{12,32}\equiv a_{32,12}$, the matrix $M \mod 2$ is either trivial or the identity, which implies that
\[
a_{12,32}\equiv a_{32,12}\equiv \Delta_{1,2,3}(a_{132}) \equiv \Delta_{3,2,1}(a_{132})\equiv 0 \mod 2.
\]
From the recursive formulas for the diagonal entries we obtain
\[
a_{\emptyset}\equiv a_{2,2}\equiv a_{12,12}\equiv a_{32,32} \equiv a_{132,132}\equiv a_{2132,2132} \equiv 0, 1 \mod 2.
\]
So $\phi$ is trivial and, hence, we obtain 
\begin{prop} Let $G$ be a split simple group of type $\Atype_3$ with the character lattice 
$T^*=\langle \alpha_1,\alpha_2,\alpha_3,\omega_2\rangle$, i.e.\  $G=SL_4/\mu_2$.  
Let $P$ be a parabolic subgroup of type $\Atype_1\times \Atype_1$.

Then the $D$-module $D_P^\star$ is indecomposable over $\Z/2\Z$.
\end{prop}

\begin{rem}
Again in view of Corollary~\ref{cor:mainindec} this fact implies indecomposability of the motive (with $\Z/2\Z$-coefficients) 
of a generic 4-dimensional quadric.
\end{rem}

\subsection{Involution varieties}
Let $G$ be a split simple group of type $\Dtype_4$ and let $P$ be of type $\Atype_3$, i.e.\  $G/P$ is a split 6-dimensional smooth projective quadric.

By definition, the Weyl group $W=\langle s_1,s_2,s_3,s_4\rangle$, $W_P=\langle s_2,s_3,s_4\rangle$ and the set of minimal coset representatives $W^P$ is given by the Hasse diagram
{\small \[
\xymatrix{
1 \ar[r]^{\scriptscriptstyle s_1\cdot } & s_1  \ar[r]^{\scriptscriptstyle s_2\cdot } & s_2s_1 \ar[d]_{\scriptscriptstyle s_3\cdot }\ar[r]^{\scriptscriptstyle s_4\cdot }  & s_4s_2s_1\ar[d]^{\scriptscriptstyle s_3\cdot } & &\\
& & s_3s_2s_1 \ar[r]^{\scriptscriptstyle s_4\cdot }& s_4s_3s_2s_1\ar[r]^{\scriptscriptstyle s_2\cdot } & s_2s_4s_3s_2s_1\ar[r]^{\scriptscriptstyle s_1\cdot } & s_1s_2s_4s_3s_2s_1
}
\]}
By the recursive formulas of Lemma~\ref{lem:recursive} and of Corollary~\ref{cor:res} we obtain (as before we set $a_{ijk..}=a_{s_is_js_k..}$ and $a_\emptyset=a_1$):
\begin{align*}
a_{1,1} &=a_\emptyset+\Delta_{1,2,3,4,2}(a_{24321}), \\
a_{21,21} &=a_{1,1}-\Delta_{2,1,3,4,2}(a_{24321}), \\
a_{321,321} &=a_{21,21}+\Delta_{3,2,1,4,2}(a_{24321})\text{ and } a_{421,421}=a_{21,21}+\Delta_{4,2,1,3,2}(a_{24321}),\\
a_{321,421} &=\Delta_{4,2,1,4,2}(a_{24321})\text{ and } a_{421,321}=\Delta_{3,2,1,3,2}(a_{24321}), \\
a_{4321,4321} &=a_{421,421}+\Delta_{3,2,1,4,2}(a_{24321})-\Delta_{3,4,2,1,2}(a_{24321}) \\
 &= a_{321,321}+ \Delta_{4,2,1,3,2}(a_{24321}) - \Delta_{4,3,2,1,2}(a_{24321}), \\
a_{24321,24321} &= a_{4321,4321}+(\Delta_{2,4,3,2,1}-\Delta_{2,4,3}s_2\Delta_{1,2})(a_{24321}), \\
a_{124321,124321} &=a_{24321,24321}+\Delta_{1,2,4,3,2}(s_1(a_{24321})+\Delta_1(a_{124321})).
\end{align*}
which can be also expressed as the following diagram of operators (here we denote $a_{ijk..,lmn..}$ by $a^{lmn..}_{ijk..}$)
{\small 
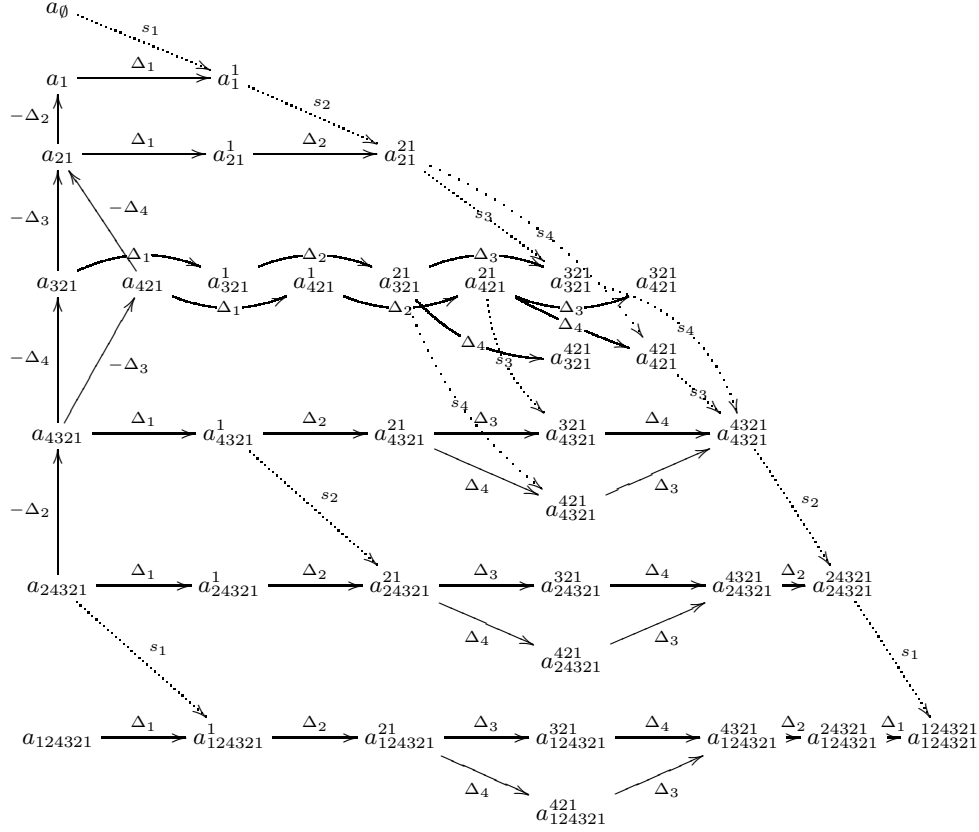
\[
\xymatrix@C=5pt@R=14pt{
a_\emptyset \ar@{.>}[rrd]^{s_1}  & & & & & & & & & & \\
a_1 \ar[rr]^{\Delta_1}  &  & a_1^1 \ar@{.>}[rrd]^{s_2} & & & & & & & & \\
a_{21} \ar[rr]^{\Delta_1} \ar[u]^{-\Delta_2} & & a_{21}^1 \ar[rr]^{\Delta_2} & & a_{21}^{21}  \ar@{.>}[rrdd]|-{s_3} \ar@{.>}@/^10pt/[rrrddd]|-{s_4} & & & & & & \\
  & & & & & & & & & & \\
a_{321} \ar@/^10pt/[rr]|-{\Delta_1} \ar[uu]^{-\Delta_3} & a_{421} \ar@/_10pt/[rr]|-{\Delta_1} \ar[uul]_{-\Delta_4}  & a_{321}^{1} \ar@/^10pt/[rr]|-{\Delta_2}  & a_{421}^1  \ar@/_10pt/[rr]|-{\Delta_2} & a_{321}^{21} \ar@/^10pt/[rr]|-{\Delta_3}\ar@/_10pt/[rrd]|-{\Delta_4} \ar@{.>}@/_10pt/[rrddd]|-{s_4} & a_{421}^{21} \ar@/_10pt/[rr]|-{\Delta_3} \ar@/_2pt/[rrd]|-{\Delta_4} \ar@{.>}@/_9pt/[rdd]|-{s_3} & a_{321}^{321} \ar@{.>}@/^15pt/[rrdd]|-{s_4} & a_{421}^{321} &  &   &  \\
& & & & & & a_{321}^{421}  & a_{421}^{421} \ar@{.>}[rd]|-{s_3}& & &  \\
a_{4321} \ar[rr]^{\Delta_1} \ar[uu]^{-\Delta_4} \ar[uur]_{-\Delta_3} & & a_{4321}^{1} \ar[rr]^{\Delta_2} \ar@{.>}[rrdd]^{s_2} & & a_{4321}^{21} \ar[rr]^{\Delta_3}\ar[rrd]_{\Delta_4} & & a_{4321}^{321} \ar[rr]^{\Delta_4} & & a_{4321}^{4321}   \ar@{.>}[rdd]^{s_2} &   &  \\
& & & & & & a_{4321}^{421} \ar[rru]_{\Delta_3} & & & &  \\
a_{24321} \ar[rr]^{\Delta_1} \ar@{.>}[rrdd]^{s_1} \ar[uu]^{-\Delta_2} & & a_{24321}^{1} \ar[rr]^{\Delta_2}& & a_{24321}^{21} \ar[rr]^{\Delta_3}\ar[rrd]_{\Delta_4} & & a_{24321}^{321} \ar[rr]^{\Delta_4} & & a_{24321}^{4321} \ar[r]^{\Delta_2} & a_{24321}^{24321}  \ar@{.>}[rdd]^{s_1} &  \\
& & & & & & a_{24321}^{421} \ar[rru]_{\Delta_3} & & & &  \\
a_{124321} \ar[rr]^{\Delta_1} & & a_{124321}^{1} \ar[rr]^{\Delta_2}& & a_{124321}^{21} \ar[rr]^{\Delta_3}\ar[rrd]_{\Delta_4} & & a_{124321}^{321} \ar[rr]^{\Delta_4} & & a_{124321}^{4321} \ar[r]^{\Delta_2} & a_{124321}^{24321} \ar[r]^{\Delta_1} & a_{124321}^{124321} \\
& & & & & & a_{124321}^{421} \ar[rru]_{\Delta_3} & & & & 
}
\]}
Let $\Delta_{i_1,i_2,..}^d$ denote the image $\Delta_{i_1,i_2,..}(Sym^d(T^*))$ modulo $2$.

\paragraph{\it $PGO_8$-case}
Assume $T^*$ corresponds to the root lattice, i.e.\ $G=PGO_8$.

\begin{lem}\label{lem:pgo8} 
We have $\Delta_{3,4}^2=0$, $\Delta_{3,4}^3=\langle \alpha_3+\alpha_4 \rangle$,
\[
\Delta_{3,4}^4=\langle (\alpha_3+\alpha_4)\alpha_1,  (\alpha_3+\alpha_4)\alpha_3, (\alpha_3+\alpha_4)\alpha_4\rangle,\quad \Delta_{2,3,4}^4=\langle \alpha_3+\alpha_4\rangle.
\]
Moreover, it holds for any permutation of the set of subscripts $\{1,3,4\}$.
\end{lem}

\begin{proof}
Follows from the fact that $\Delta_3$ and $\Delta_4$ are trivial mod 2 on all simple roots except $\alpha_2$ and that $\Delta_{3,4}(\alpha_2^2)\equiv \Delta_{3,4}(\alpha_2^4)\equiv 0$,
$\Delta_{3,4}(\alpha_2^3)\equiv \alpha_3+\alpha_4$.
\end{proof}

From the lemma we immediately obtain 
\begin{align*}
a_{124321,124321}, a_{1,1}: & \Delta_{1,2,3,4}^4=\Delta_{1}(\Delta_{2,3,4}^4)=0 \\
a_{24321,24321}, a_{21,21}: & \Delta_{2,3,4}^3=0, \; \Delta_{2,1}(\Delta_{3,4}^4)=0 \\
a_{321,321}, a_{421,421}: & \Delta_{3,2,1,4}^4=0,\; \Delta_{4,2,1,3}^4=0 \\
a_{4321,4321}: & \Delta_{4,2,1,3}^4=0,\; \Delta_{3,4}^2=0 \\
\end{align*}
If $\phi\in E_D$ is an idempotent, it gives 
\[
a_\emptyset \equiv a_{1,1}\equiv a_{21,21}\equiv a_{321,321}\equiv a_{421,421} \equiv a_{4321,4321}\equiv a_{24321,24321}\equiv a_{124321,124321}.
\]
So there are no non-trivial idempotents
and we obtain
\begin{prop} Let $G=PGO_8$ and let $P$ be the maximal parabolic subgroup generated by all simple reflections except the first one. 

Then the $D$-module $D_P^\star$ is indecomposable over $\Z/2\Z$.
\end{prop}

\begin{rem}
In view of Corollary~\ref{cor:mainindec} this fact implies indecomposability of the motive of a generic twisted form (e.g. of involution variety)
of a 6-dimensional split quadric.
\end{rem}

\paragraph{\it $SO_8$-case}
Assume that $\omega_1\in T^*$, that is $G=SO_8$.
Then Lemma~\ref{lem:pgo8} turns into
\begin{lem} We have $\Delta_{3,4}^2=0$, $\Delta_{3,4}^3=\langle \alpha_3+\alpha_4 \rangle$,
\[
\Delta_{3,4}^4=\langle (\alpha_3+\alpha_4)\alpha_1,  (\alpha_3+\alpha_4)\alpha_3, (\alpha_3+\alpha_4)\alpha_4, (\alpha_3+\alpha_4)\omega_1\rangle,\quad \Delta_{2,3,4}^4=\langle \alpha_3+\alpha_4\rangle.
\]
\end{lem}
So we obtain
\begin{align*}
a_{124321,124321}, a_{1,1}: & \Delta_{1,2,3,4}^4=\Delta_{1}(\Delta_{2,3,4}^4)=0 \\
a_{24321,24321}, a_{21,21}: & \Delta_{2,3,4}^3=0, \; \Delta_{2,1}(\Delta_{3,4}^4)=0 
\end{align*}
which gives only that
\[
a_\emptyset \equiv a_{1,1}\equiv a_{21,21}\quad \text{ and }\quad a_{4321,4321}\equiv a_{24321,24321}\equiv a_{124321,124321}.
\]

Since $\Delta_3(f)\equiv \Delta_4(f)$ for any linear $f$, we have
\[
\Delta_{4,2,1,3,2}\equiv \Delta_{3,2,1,3,2}.
\]
Moreover, direct computations show that
\[
\Delta_{3,2,3}(\alpha_2^2\alpha_3)\equiv \Delta_{3,2,3}(\alpha_2^2\alpha_4)\equiv 1.
\]
So that $\Delta_{3,2,3}(g) \equiv \Delta_{4,2,4}(g)$.
Combining, we obtain
\[
\Delta_{4,2,1,3}= \Delta_{4,2,3,1}\equiv \Delta_{3,2,3,1}\equiv \Delta_{4,2,4,1}\equiv \Delta_{3,2,4,1} =\Delta_{3,2,1,4}
\]
So $a_{21,21}\equiv a_{321,321}\equiv a_{421,421}\equiv a_{4321,4321}$.
Hence, there are no non-trivial idempotents as well and we get

\begin{prop} Let $G=SO_8$ and let $P$ be the maximal parabolic subgroup generated by all simple reflections except the first one. 

Then the $D$-module $D_P^\star$ is indecomposable over $\Z/2\Z$.
\end{prop}

\begin{rem}
This fact implies indecomposability of the motive of a generic 6-dimensional quadric.
\end{rem}

\paragraph{\it $HSpin_8$-case}
Assume $\omega_4\in T^*$ that is $G=HSpin_8$. Then $T^*=\langle \alpha_2,\alpha_3,\alpha_4,\omega_4\rangle$.

We claim that $\Delta_{1,2,3,4,2}(a_{24321})\equiv 0$. 
Indeed, let $f=\Delta_{2,3,4,2}(a_{24321}) \in Sym^1(T^*)$. Then 
$\Delta_3(f)=\Delta_{3,2,3,4,2}(a_{24321})=\Delta_{2,3,2,4,2}(a_{24321})=\Delta_{2,3,4,2}(\Delta_4(a_{24321}))=0$.
Now for $f=a_2\alpha_2+a_3\alpha_3+a_4\alpha_4+b\omega_4$ we get $\Delta_1(f)\equiv a_2\mod 2$ but
$\Delta_3(f)\equiv a_2\mod 2$ as well.

Similarly, $\Delta_{3,2,1,4,2}(a_{24321})\equiv 0$. In this case denote $f=\Delta_{2,1,4,2}(a_{24321})$. Then
$\Delta_1(f)=\Delta_{1,2,1,4,2}(a_{24321})=\Delta_{2,1,2,4,2}(a_{24321})=\Delta_{2,1,4,2}(\Delta_4(a_{24321}))=0$.
And $\Delta_3(f)\equiv \Delta_1(f)\equiv 0$.

By the same arguments, $\Delta_{3,2,1,3,2}(a_{24321})\equiv \Delta_{3,4,2,1,2}\equiv 0$.

Consider now $\Delta_{2,1,3,4,2}(a_{24321})$. Let $g=\Delta_{3,4,2}(a_{24321})$. We have $\Delta_{3,2}(g)=0$.
Let $g=\sum_{2\le i\le j} c_{ij}\alpha_i\alpha_j+\sum_{2\le i} b_i\omega_4\alpha_i+d \omega_4^2$.
Then 
\[
\Delta_2(g)\equiv c_{22}(\alpha_2+\alpha_3+\alpha_4)+\alpha_2(c_{23}+c_{24})+\omega_4(b_3+b_4).
\]
The fact that $\Delta_{3,2}(g)\equiv 0$ implies that $c_{22}+c_{23}+c_{24}\equiv 0$. But
\[
\Delta_1(g)\equiv c_{22}\alpha_1+c_{23}\alpha_3+c_{24}\alpha_4+b_2\omega_4.
\]
So that $\Delta_{2,1}(g)\equiv (c_{22}+c_{23}+c_{24})\equiv 0$.
Combining we obtain that
\[
a_\emptyset\equiv a_{1,1}\equiv a_{21,21}\equiv a_{321,321}\equiv a_{421,321}
\]
and
\[
a_{421,421}\equiv a_{321,421}\equiv a_{4321,4321}\equiv a_{24321,24321}\equiv a_{124321,124321}.
\]

So the $D$-module $D_P^\star$ is either irreducible or it splits into two indecomposable direct summands with  a generating function $1+t+t^2+t^3$ (over $S$) each.

\begin{rem}\label{rem:hspin}
This fact implies that the motive of a $HSpin_8$-versal involution variety $E/P$ 
is either indecomposable or splits as a direct sum $[E/P]=N\oplus N(3)$, where $N$ is indecomposable with a generating function
$1+t+t^2+t^3$. Using know result on motives of quadratic forms  (e.g. that after splitting the algebra, the motive of a $Spin_8$-versal quadric splits into 2-fold Rost motives~\cite[XVII]{EKM}) it follows that the second decomposition is impossible, i.e.\ $[E/P]$ is indecomposable.
\end{rem}


\appendix

\section{Generic nilpotence for cobordism}

This appendix is included to prove Lemma~\ref{lem:hintersect}, the obvious generalization from Chow groups to algebraic cobordism of \cite[Lemma~6.2]{VZ}. We therefore closely follow arguments of \cite{VZ}, while making the necessary adjustments to cobordism. 

Let $\Sch_k$ denote the category of reduced schemes of finite type over $k$ and let $\Sm_k$ denote its subcategory of smooth schemes over $k$. Let $\Omega_*(-)$ denote the algebraic cobordism functor of Levine-Morel \cite{LM}.

\begin{dfn}
Let $X \in \Sch_k$ and $i \colon Z\hookrightarrow X$ be a closed subset. We will denote the image of the map $i_*\colon\Omega_*(Z)\to\Omega_*(X)$ by $\Omega_Z(X)$ and say that elements of $\Omega_Z(X)$ are supported on $Z$.
\end{dfn}

\begin{dfn}\label{dfn:gcover} 
We will call a projective map $f\colon Z'\to Z$ in $\Sch_k$ a \emph{good cover} if there are filtrations by closed subsets $\emptyset\subseteq Z'_0\subseteq Z'_1\subseteq\ldots\subseteq Z'_n=Z'$ and $\emptyset\subseteq Z_0\subseteq Z_1\subseteq\ldots\subseteq Z_n=Z$ such that
\begin{itemize}
\item $\dim Z'_i=\dim Z_i$, 
\item $f^{-1}(Z_i-Z_{i-1})\to Z_i-Z_{i-1}$ is an isomorphism and
\item $f(Z'_i)\subseteq Z_i$.
\end{itemize}
for any $0\leqslant i\leqslant n$.
\end{dfn}

\begin{ex}
Here is a typical good cover: If $f\colon X'\to X$ is a projective map that induces an isomorphism of non-empty open subsets $f^{-1}(X-Z)\cong X-Z$ for some closed subset $Z$ of $X$ with $\dim Z < \dim X$, then the map $X'\sqcup Z\to X$ is a good cover: the respective filtrations can be chosen as $\emptyset\subseteq Z\subseteq X'\sqcup Z$ and $\emptyset\subseteq Z\subseteq X$.
\end{ex}

\begin{rem} \label{rem:gcover}
We could have defined good covers inductively by declaring that isomorphisms are good covers and that the map $f\colon X'\to X$ in $\Sch_k$ is a good cover provided there are closed subsets $Z\subseteq X$ and $Z'\subseteq X'$ such that 
\begin{itemize}
\item $\dim X' = \dim X$,
\item $f^{-1}(X-Z)\to X-Z$ is an isomorphism,
\item $f(Z')\subseteq Z$ and $f_{|Z'}\colon Z'\to Z$ is a good cover. 
\end{itemize}
\end{rem}

\begin{lem}\label{gcoversurjects} 
If $f\colon Z'\to Z$ is a good cover, then $f_*\colon \Omega_*(Z')\to\Omega_*(Z)$ is surjective.
\end{lem}

\begin{proof} 
Let $(Z_i)_{i=0}^n$ and $(Z'_i)_{i=0}^n$ be filtrations as in Definition~\ref{dfn:gcover}. We prove that $\Omega_*(Z'_i)\to\Omega_*(Z_i)$ is surjective by induction on $i$. The case $i=0$ is obvious since $Z'_0\cong Z_0$. For the induction step consider the map of localization sequences induced by $f_*$:
\[
\xymatrix{
\Omega_*(f^{-1}(Z_{i-1}))\ar[r]\ar[d] & \Omega_*(Z'_i)\ar[r]\ar[d] & \Omega_*(f^{-1}(Z_i-Z_{i-1}))\ar[d]\ar[r] & 0\\
\Omega_*(Z_{i-1})\ar[r] & \Omega_*(Z_{i})\ar[r] & \Omega_*(Z_i-Z_{i-1})\ar[r] & 0
}
\]
The right vertical map is an isomorphism by Definition~\ref{dfn:gcover}. The left vertical map is surjective because it fits into the composition
\[
\Omega_*(Z'_{i-1})\to\Omega_*(f^{-1}(Z_{i-1}))\to\Omega_*(Z_{i-1})
\]
where the composite map is surjective by the induction hypothesis. Hence the central arrow is surjective as well.
\end{proof}

\begin{lem}\label{lem:gcoverpb} 
Good covers are preserved by flat base-change: if $f\colon Z'\to Z$ is a good cover and $g\colon Y\to Z$ is flat then $f'\colon Y\times_Z Z'\to Y$, the base-change of $f$ to $Y$, is a good cover.
\end{lem}

\begin{proof}
Let $(Z'_i)_{i=1}^n$ and $(Z_i)_{i=1}^n$ be filtrations on $Z'$ and $Z$ as in Definition~\ref{dfn:gcover}, and let $Y_i=g^{-1}(Z_i)$ and $Y'_i=(g')^{-1}(Z'_i)$. These are filtrations on $Y$ and $Y'$ that satisfy Definition~\ref{dfn:gcover}. Indeed, we have $\dim Y'_i=\dim Y_i$ by flatness of $g$ (and $g'$), the map $(f')^{-1}(Y_i-Y_{i-1})\to Y_i-Y_{i-1}$ is the base change of the isomorphism $f^{-1}(Z_i-Z_{i-1}) \to Z_i-Z_{i-1}$ so it is an isomorphism, and the inclusion $f'(Y'_i)\subseteq Y_i$ is obvious.
\end{proof}

\begin{lem}\label{lem:splitting}
Let $X\in\Sch_k$ and let $E$ be a vector bundle over $X$. Suppose that $Z$ is a closed subset of $X$. Then there exists a projective map $p\colon X'\to X$ with $\dim X'=\dim X$ and a closed subset $Z'$ of $X'$ such that 
\begin{enumerate}
\item \label{item:subbundles} $p^*E$ has a filtration by subbundles whose successive quotients are line bundles;
\item \label{item:good} $p(Z')\subseteq Z$, and the map $p\colon Z'\to Z$ is a good cover.
\end{enumerate} 
\end{lem}

\begin{proof}
Consider the projection $p\colon \Fl(E)\to X$ where $\Fl(E)$ is the variety of complete flags of the vector bundle $E$ over $X$. Then $p^*E$ has a filtration by tautological subbundles. Any point $x\in X$ is contained in an open neighbourhood $U$ over which the bundle $E$ is trivial, so $\Fl(E)|_{U}\simeq \Fl\times U$ and there is a section $s\colon U\to p^{-1}(U)$. Let $X'_U$ be the closure of $s(U)$ in $\Fl(E)$. Note that the projection $(p_{|X'_U})^{-1}(U) \to U$ is an isomorphism.

We will construct $X'$ as a disjoint union of schemes of the form $X'_U$ by induction on $d=\dim Z$:

When $d=0$, the set $Z$ is a finite union of closed points $z_1,\ldots,z_n$. Each $z_i$ is contained in some neighbourhood $U_i$ as above and we can take $X'=X'_{U_1}\sqcup\ldots\sqcup X'_{U_n}$. Each $X'_{U_i}$ contains a copy of the point $z_i$, and we can take $Z'$ to be the disjoint union of those copies. Then $Z'=Z$ is a good cover.

Consider the case of a general $d$. Let $Z_1,\ldots,Z_n$ be irreducible components of $Z$. For any $i$ there is a neighbourhood $U_i$ of the generic point of $Z_i$  such that $Z_i\cap U_i$ is disjoint with $Z_j$ for $j\neq i$, and $E$ is trivial over $U_i$. Then let $Z'_i$ be the closure of $Z_i\cap U_i$ in $X'_{U_i}$. Observe that the projection $p_{|Z'_i}\colon Z'_i\to Z_i$ induces an isomorphism $(p_{|Z'_i})^{-1}(Z_i\cap U_i)\to Z_i\cap U_i$.

Let $D_i=Z_i-(Z_i\cap U_i)$ and $D=\cup_{i=1}^n D_i$. Then $\dim D<d$ and by the induction hypothesis there is a map $p_D\colon X'_D\to X$ and a closed subset $Z'_D\subseteq X'_D$ such that $p_D\colon Z'_D\to D$ is a good cover, and $p_D^*(E)$ satisfies~\eqref{item:subbundles}. Take 
\[
X'=X'_D\sqcup X'_{U_1}\ldots\sqcup X'_{U_n}, Z'=Z'_D\sqcup Z'_1\sqcup\ldots\sqcup Z'_n.
\]
Consider the map $p_{|Z'}\colon Z'\to Z$. We have $Z-D=\sqcup_i (Z_i\cap U_i)$, hence $(p_{|Z'})^{-1}(Z-D)\to Z-D$ is an isomorphism. By Remark~\ref{rem:gcover}, the map $p_{|Z'}$ is a good cover, so $X'$ and $Z'$ satisfy conditions~\eqref{item:subbundles} and~\eqref{item:good}.
\end{proof}

\begin{lem}\label{lem:divsupport}
Let $X\in\Sch_k$ and let $E$ be a vector bundle over $X$ that has a filtration by subbundles with line bundle quotients $E_1,\ldots,E_d$. Assume that $D_1,\ldots,D_d$ are divisors such that $E_i$ is trivial over $X-D_i$ for each $i$. Let $D=D_1\cap\cdots\cap D_d$. Then for any $\alpha\in\Omega_Z(X)$ the element $\widetilde{c}_d(E)\cap\alpha$ lies in $\Omega_{Z\cap D}(X)$. 
\end{lem}

\begin{proof}
By the Whitney formula, we have $\widetilde{c}_d(E)=\widetilde{c}_1(E_1)\circ\ldots\circ\widetilde{c}_1(E_d)$. For any $\alpha\in\Omega_Z(X)$ the element $\widetilde{c}_1(E_d)\cap\alpha$ vanishes in $\Omega_{Z-D_d}(X-D_d)$, hence it lies in $\Omega_{Z\cap D_d}(X)$. The statement thus follows by induction on $d$.
\end{proof}

\begin{lem}\label{lem:chernintersection}
Let $E$ be a vector bundle of rank $d$ on a quasi-projective variety $X$ and let $Z$ be a closed subset of $X$.
Then there is a closed subset $\widetilde{Z}$ of $Z$ such that $\dim\widetilde{Z}\leqslant\dim Z-d$ and for any smooth morphism $f\colon Y \to X$ of schemes in $\Sch_k$ and any $\alpha\in\Omega_{f^{-1}(Z)}(Y)$ the element $\widetilde{c}_d(f^*E)\cap\alpha$ is supported on $f^{-1}(\widetilde{Z})$.
\end{lem}

\begin{rem}
If $\dim Z -d<0$, then $\dim \widetilde{Z}<0$ so $\widetilde{Z}=\emptyset$, and the lemma concludes that $\widetilde{c}_d(f^*E)\cap\alpha=0$ as already known by \cite[2.3.13]{LM}. 
\end{rem}

\begin{proof}
By Lemma~\ref{lem:splitting} there is a projective map $p\colon X'\to X$ and a closed subset $Z'\subseteq X$ such that $p_{|Z'}\colon Z'\to Z$ is a good cover while $p^*E$ has a filtration with line bundle quotients $E_1,\ldots, E_d$. Using quasi-projectivity of $X$, let $D_1,\ldots, D_d$ denote divisors such that each $E_i$ is trivial over $X'-D_i$. Set $D=D_1\cap \cdots \cap D_d$. One can furthermore ensure that $Z'\cap D$ and thus $\widetilde{Z}=p(Z'\cap D)$ has dimension at most $\dim Z'-d=\dim Z -d$. Let $Y'=X'\times_X Y$ and let $g\colon Y'\to X'$ be the canonical projection. Then $g^{-1}(Z')\to f^{-1}(Z)$ is a good cover by~\ref{lem:gcoverpb}. By Lemma~\ref{gcoversurjects} the projection $q\colon Y'\to Y$ induces the surjective push-forward $\Omega_*(g^{-1}(Z')) \to \Omega_*(f^{-1}(Z))$ and thus $\Omega_{g^{-1}(Z')}(Y')\to\Omega_{f^{-1}(Z)}(Y)$ is also surjective. Given $\alpha\in \Omega_{f^{-1}(Z)}(Y)$, choose a preimage $\alpha'\in\Omega_{g^{-1}(Z')}(Y')$. Then $\widetilde{c}_d(f^*E)\cap\alpha=q_*(\widetilde{c}_d(q^*f^*E)\cap\alpha')$, and $\widetilde{c}_d(q^*f^*E)\cap\alpha'$ is supported on $g^{-1}(Z'\cap D)$ by Lemma~\ref{lem:divsupport}, thus $\widetilde{c}_d(f^*E)\cap\alpha$ is supported on $q(g^{-1}(Z'\cap D))$ which is contained in $f^{-1}(\widetilde{Z})$.
\end{proof}

\begin{lem}(cf.\ \cite[Lemma 6.3]{VZ})\label{lem:intersect}
Let $f\colon V\hookrightarrow B$ and $g\colon T \hookrightarrow B$ be closed embeddings with regular $f$ and smooth quasi-projective $B$. Then there exists a closed embedding $h\colon Z\hookrightarrow V$ such that $\codim h\geqslant\codim g$, and for any smooth morphism $\varepsilon\colon W\to B$ we have $\im \big(f_W^*\circ (g_W)_*\big)\subseteq \im \big((h_W)_*\big)$ inside $\Omega_*(V_W)$, where $f_W$, $g_W$ and $h_W$ are the base changes of $f$, $g$ and $h$ to $W$ along $\varepsilon$.
\end{lem}

\begin{proof}
Consider the intersection $\widetilde{T}$ of $T$ and $Z$ in $B$ and base-change it along $\varepsilon$ to obtain the Cartesian squares
\[
\begin{tikzcd}
T\ar[r,hook,"g"] & B\\
\tilde{T}\ar[r,hook,"\tilde{g}"] \ar[u,hook,"\tilde{f}"] & V\ar[u,hook,"f"]
\end{tikzcd}
\qquad\text{and}\qquad
\begin{tikzcd}
T_W\ar[r,hook,"g_W"] & W\\
\widetilde{T}_W \ar[u,hook,"\widetilde{f}_W"] \ar[r,hook,"\widetilde{g}_W"] & V_W\ar[u,hook,"f_W"]
\end{tikzcd}
\]
By~\cite[Proposition 6.6.3]{LM} $f_W^*\circ (g_W)_*=(\widetilde{g}_W)_*\circ f_W^!$ where the refined pullback $f_W^!$ is given by the composition~\cite[(6.17), \S 6.6.2]{LM}
\[
\Omega_*(T_W)\to \Omega_*(C_W)\to \Omega_*(N_W)\to \Omega_{*-d}(\widetilde{T}_W)
\]
where $d$ is the codimension of $f$, while $C_W$ is the normal cone of $\widetilde{f}_W$ and $N_W=\widetilde{g}_W^*(N_{f_W})$ is the the normal bundle of $f_W$ pulled back to $\widetilde{T}_W$.

Let $N=\widetilde{g}^*(N_f)$ be the pullback of the normal bundle of $f$ and $C$ be the normal cone of the map $\widetilde{f}$, both being bundles over $\widetilde{T}$. Let $q\colon\PP(N\oplus 1)\to \widetilde{T}$ be the structural map. Note that by flatness of $\varepsilon\colon W \to B$, base changing $C$ (resp.\ $N$) to $W$ does yield $C_W$ (resp.\ $N_W$). Applying Lemma~\ref{lem:chernintersection} to the vector bundle $E=q^*N\otimes\Os(1)$ and to the closed subset $\PP(C\oplus 1)$ inside $\PP(N\oplus 1)$, we get a closed subset $\widetilde{Z}$ of codimension at least $d$ in $\PP(C\oplus 1)$ such that for every smooth morphism $\epsilon\colon W \to B$ and every cobordism class $x$ supported on $\PP(C_W\oplus 1)$ the class $\widetilde{c}_d(\varepsilon^*E)\cap x$ is supported on $\varepsilon^{-1}(\widetilde{Z})$. Set $Z=q(\widetilde{Z})$.
By Lemma~\ref{lem:zeropullback} below, the right hand side of the diagram
\[
\begin{tikzcd}[column sep=large]
\Omega_*\big(\PP(C_W\oplus 1)\big) \ar[d,twoheadrightarrow] \ar[r] & \Omega_*\big(\PP(N_W\oplus 1)\big) \ar[d,twoheadrightarrow] \ar[r,"\widetilde{c}_d(\varepsilon^*E)\cap - "] & \Omega_*\big(\PP(N_W\oplus 1)\big) \ar[d] \\
\Omega_*(C_W) \ar[r] & \Omega_*(N_W) \ar[r] & \Omega_*(\widetilde{T}_W)
\end{tikzcd}
\]
is commutative (and so is the left hand-side), while the left vertical map is surjective since $C_W$ is open in $\PP(C_W\oplus 1)$. Therefore, the image of the composition $\Omega_*(C_W)\to \Omega_*(N_W)\to \Omega_{*-d}(\widetilde{T}_W)$ is supported on $q(\varepsilon^{-1}(\widetilde{Z}))\subseteq \varepsilon^{-1}(Z)=Z_W$, i.e.\ is in $\Omega_{Z_W}(\widetilde{T}_W)$ which maps to $\Omega_{Z_W}(V_W)=\im\big((h_W)_*\big)$. Thus, $h\colon Z\hookrightarrow V$ satisfies the requirements.
\end{proof}

\begin{lem}\label{lem:zeropullback}
Let $X\in\Sch_k$ and let $E\to X$ be a rank $d$ vector bundle with a zero section $z\colon X\to E$. Let $q\colon \PP(E \oplus 1) \to X$ be the structural map and let $a\colon E \to \PP(E \oplus 1)$ be the natural open affine chart sending $e$ to $[e:1]$. Then the following diagram commutes.
\[
\xymatrix{
\Omega_*(\PP(E\oplus 1))\ar[d]_{a^*} \ar[rrr]^{\widetilde{c}_d(q^*E\otimes\Os(1))\cap-} &&& \Omega_{*-d}(\PP(E\oplus 1))\ar[d]^{q_*}\\
\Omega_*(E)\ar[rrr]^{z^*} &&& \Omega_{*-d}(X)
}
\]
\end{lem}

\begin{proof}
The composition $\Os(-1)\to q^*(E\oplus 1)\to q^*E$ of the natural embedding and projection defines a global section $s$ of the sheaf $q^*E\otimes\Os(1)=\underline{\Hom}(\Os(-1),q^*E)$. The zero set of $s$ is $X \simeq \PP(1)$ (regularly) embedded in $\PP(E \oplus 1)$, and this embedding coincides with $\bar{s}=a \circ z$. By~\cite[Lemma~6.6.7]{LM}, the operator $\widetilde{c}_d(q^*E\otimes\Os(1))\cap -$ on $\Omega_*(\PP(E\oplus 1))$ is given by $\bar{s}_*\bar{s}^*$. The right-down composition is thus $q_*\bar{s}_*\bar{s}^*=\bar{s}^*=z^* a^*$.
\end{proof}

Let $\hh$ be an oriented cohomology theory that is generically constant (i.e.\  the pull-back $\hh(k) \to \hh(L)$ is an isomorphism for any field extensions $L/k$) and the canonical map from algebraic cobordism $\Omega(X)\otimes_{\mathbb{L}}\hh(k)\to\hh(X)$ is surjective for every $X\in\Sm_k$. 
An example of such theory is given by the free theory of \cite[\S2.12]{GV}.

\begin{lem}\label{lem:hintersect}
Let $X$ be a smooth quasi-projective variety. Let $i_1\colon Z_1\hookrightarrow X$ and $i_2\colon Z_2\hookrightarrow X$ be closed embeddings. Then there exists a closed embedding $i_3\colon Z_3\hookrightarrow X$ such that for any smooth morphism $\pi\colon Y\to X$,
\[
\codim Z_3\geqslant \codim Z_1+\codim Z_2
\quad\text{and}\quad
\im(i'_1)_*\cdot\im(i'_2)_*\subseteq im(i'_3)_*\text{ in }\hh(Y),
\]
where $i'_j\colon Z_j\times_X Y\to Y$, $j=1,2,3$ are the base change along $\pi$ of the respective closed embeddings.
\end{lem}

\begin{proof}
We first consider the case $\hh=\Omega$. Applying Lemma~\ref{lem:intersect} to $B=X \times X$, $f=\Delta_X\colon X \to X \times X$, $g=i_1 \times i_2\colon Z_1 \times Z_2 \to X \times X$, one obtains a closed embedding $h\colon Z \hookrightarrow X$ such that $\codim Z\geqslant \codim Z_1+\codim Z_2$. Choosing $W=Y\times Y$ and $\varepsilon=\pi \times \pi\colon Y \times Y \to X \times X$, it follows that $f_W$ is the natural inclusion $Y \times_X Y \to Y \times Y$, $h_W\colon  Z_{Y \times Y} \hookrightarrow Y\times_X Y$ is the obvious map and 
\[
\im \big(f_W^*\circ(i'_1\times i'_2)_*\big)\subseteq \im \big((h_W)_*\big).
\]
The diagonal embedding $Y\to Y\times Y$ factors as $Y\stackrel{\phi}\to Y\times_X Y\stackrel{f_W}\to Y\times Y$. Using the Cartesian square
\[
\begin{tikzcd}
Y \ar[r,"\phi"] & Y\times_X Y \\
Z_Y \ar[r,"\phi_Z"] \ar[u,"h_Y"] & Z_{Y\times Y} \ar[u,"h_W"]
\end{tikzcd}
\]
and \cite[Proposition 6.6.3]{LM}, we obtain $\phi^*\circ (h_W)_*=(h_Y)_*\circ \phi_Z^{!}$,
\[
\im(i'_1)_*\cdot\im(i'_2)_*= \im\big( \Delta_Y^*\circ (i_1'\times i_s')_* \big) \subseteq \im \big((h_Y)_*\big).
\]
Thus, $i_3=h$ satisfies the requirements. Finally, the natural map $\Omega_*(-)\otimes_{\mathbb{L}}\hh(k)\to\hh(-)$ is surjective and compatible with push-forwards and the intersection product, so we can replace $\Omega$ by $\hh$.
\end{proof}

\bibliographystyle{plain}

\end{document}